\documentclass{amsart}

\usepackage{amsfonts,amsmath,amssymb}
\usepackage{latexsym,amscd,graphicx}

\newcommand{\bvp}{{\overline{\vp}}}

\newcommand{\wG}{{\widehat{\mathcal{G}}}}

\newcommand{\q}{u(g)^{-1}}

\newcommand{\wh}[1]{\widehat{#1}}

\newcommand{\ve}{\varepsilon}
\newcommand{\vp}{\varphi}
\newcommand{\ld}{\ldots}

\newcommand{\diag}{\mathrm{diag}}

\DeclareMathOperator*{\ot}{\otimes}
\DeclareMathOperator*{\op}{\oplus}

\newcommand{\beq}{\begin{equation}}
\newcommand{\eeq}{\end{equation}}
\newcommand{\beas}{\begin{eqnarray*}}
\newcommand{\eeas}{\end{eqnarray*}}
\newcommand{\id}{\mathrm{id}}

\newcommand{\cA}{\mathcal{A}}
\newcommand{\cB}{{\mathcal B}}
\newcommand{\cC}{\mathcal{C}}
\newcommand{\cD}{\mathcal{D}}

\newcommand{\cG}{\mathcal{G}}

\newcommand{\cK}{\mathcal{K}}

\newcommand{\cP}{\mathcal{P}}

\newcommand{\cR}{\mathcal{R}}
\newcommand{\cS}{\mathcal{S}}
\newcommand{\cT}{\mathcal{T}}

\newcommand{\fG}{\mathfrak{G}}
\newcommand{\fD}{\mathfrak{D}}
\newcommand{\fT}{\mathfrak{T}}
\newcommand{\fS}{\mathfrak{S}}

\newcommand{\R}{\mathbb{R}}
\newcommand{\C}{\mathbb{C}}

\newcommand{\NN}{\mathbb{N}}
\newcommand{\ZZ}{\mathbb{Z}}

\newcommand{\FF}{\mathbb{F}}

\newcommand{\charac}{\mathrm{char}}

\DeclareMathOperator{\End}{\mathrm{End}}

\DeclareMathOperator{\Ker}{\mathrm{Ker}\,}
\DeclareMathOperator{\rank}{\mathrm{rank}\,}

\DeclareMathOperator{\supp}{\mathrm{Supp}\,}
\DeclareMathOperator{\Supp}{\mathrm{Supp}}

\newcommand{\Sl}{\mathfrak{sl}}

\newcommand{\GL}{\mathrm{GL}}

\newcommand{\Span}{\mathrm{Span}}

\newtheorem{theorem}{Theorem}[section]

\newtheorem{lemma}[theorem]{Lemma}
\newtheorem{corollary}[theorem]{Corollary}
\newtheorem{proposition}[theorem]{Proposition}
\newtheorem{definition}[theorem]{Definition}

\newtheorem{remark}[theorem]{Remark}

\newcommand{\Q}{\mathfrak{Q}}

\begin{document}

\title{Group gradings and actions of pointed Hopf algebras}

\author[Bahturin]{Yuri Bahturin}
\address{Department of Mathematics and Statistics, Memorial
University of Newfoundland, St. John's, NL, A1C5S7, Canada}
\email{bahturin@mun.ca}

\author[Montgomery]{Susan Montgomery}
\address{Department of Mathematics, University of Southern California}
\email{smontgom@usc.edu}

\thanks{The first author acknowledges a partial support by NSERC Discovery Grant 2019-05695. The first and the second authors acknowledge support by Women in Science grant at USC}

\begin{abstract}
We study actions of pointed Hopf algebras on matrix algebras. Our approach is based on known facts about group gradings of matrix algebras \cite{BSZ01} and other sources.
\end{abstract}

\maketitle

\section{Introduction}\label{s1}

Many papers have been devoted to the study of actions of Hopf algebras on algebras.
However very few of them have actually classified the possible actions of a certain Hopf
algebra on a given algebra. One of the first of these papers was the determination of
the actions of the Taft algebra and its double on a one generator algebra  \cite{MS}. More recently
there has been a series of papers, such as \cite{EW,CE} on mostly semisimple Hopf algebras
acting on certain rings, usually domains. Finally in \cite{G}, the actions of a Taft algebra on
finite-dimensional algebras have been studied, in order to look at their polynomial identities.

In this paper we study the actions of finite-dimensional pointed Hopf algebras with abelian group of group-like elements on matrix algebras. Our methods are novel, in that we use very strongly the
classification of group gradings on matrices, as in \cite{BSZ01}. We also use heavily the well-known fact (see \cite{AM} and also \cite[Chapter 6]{SM}) that any action of a Hopf algebra on a central simple algebra is inner. One motivation for this work is to try to extend  the 
work of \cite{VOZ} and \cite{CC} on the Hopf Brauer group to Hopf algebras which are not quasitriangular''.

To state our results in more detail, we need few standard general definitions.

Let $H$ be a (finite-dimensional) Hopf algebra over a field $\FF$. Suppose that $H$ acts on a unital algebra $\cA$. This means that there is a unital left $H$-module structure on $\cA$ written by $(h,a)\to h*a$ such that $h\ast 1=\ve(h)1$ and $h\ast(ab)=(h_{(1)}\ast a)(h_{(2)}\ast b)$ where $h\in H$, $a,b\in \cA$ and $\Delta(h)=h_{(1)}\otimes h_{(2)}$. If $\circ$ is another action of $H$ on $\cA$, we say that the actions are \textit{isomorphic }if there is an algebra automorphism $\vp:\cA\to \cA$ such that $h\circ\vp(a)=\vp(h\ast a)$ for any $h\in H$ and any $a\in \cA$. The actions are called \textit{equivalent} if there is an automorphism of algebras $\vp:\cA\to \cA$ and an automorphism of Hopf algebras $\alpha:H\to H$ such that $\alpha(h)\circ\vp(a)=\vp(h\ast a)$ for any $h\in H$ and any $a\in \cA$.

We say that the action is \textit{inner} if there is a convolution invertible linear map $u: H\to \cA$ such that, for any $h\in H$ and $a\in \cA$ we have 

\[
h\ast a =\sum_{(h)} u(h_{(1)})au^{-1}(h_{(2)}).
\]

In Section \ref{ssGG}, we recall the classification of gradings by abelian groups on the matrix algebras. In the case where the group $\cG$ of group-likes of a Hopf algebra $H$ is semisimple, the action of $\cG$ is equivalent to the grading by the dual group $\fG$. Since any finite-dimensional pointed Hopf algebra $H$ with abelian group $\cG$ of group likes is generated by the group like and skew-primitive elements (see the survey \cite{AGI}, to determine the action of $H$ we need to know the action of skew-primitive elements on $\fG$-graded algebras. Some information about this is given in Proposition \ref{p99}.

In Section \ref{sACSA} we give the criterion of isomorphism of two actions of a finite-dimensional pointed Hopf algebras on a matrix algebra (Theorem \ref{liso_inn}) in terms of their respective inner actions.

In Section \ref{PHARO} we recall the definition of a class of pointed Hopf algebras by Anruskiewitshch - Schneider \cite{AS} in terms of generators and defining relations. Then, in Theorems \ref{pCR}, \ref{lSC}, \ref{pij-ji} we determine the relations satisfied by the matrices of the inner actions of these generators. This enables us, in what follows, to approach the classification of the actions, up to isomorphism. 

In Section \ref{ssMFCGGL} we determine the canonical form of these matrices in the case of pointed Hopf algebras of rank 1, with the cyclic group of group-likes. 

In Section \ref{ssAHDDG} we study the actions of a pointed Hopf algebra $H$ on a matrix  algebra $\cA$ in the case where the respective $\cG$-grading is a division grading, that is, $\cA$ is a $\cG$-graded division algebra. We call such actions division actions. Theorem \ref{pR1HHADG} treats division actions in the case where $H$ is a pointed Hopf Algebra of rank 1. In Proposition \ref{pp3} we classify, up to isomorphism, the division actions of a pointed Hopf algebra of dimension $p^3$ from \cite{AS}.

In Section \ref{pmixed}, we  treat a much harder case of mixed gradings. 

In Section \ref{TOM} we deal with the actions of the Taft algebra $T_n(\omega)$ canonically generated by $g,x$ on the matrix algebras. Theorem \ref{tATA} gives the classification of such actions in the case where the matrix $u(x)$ of the inner action of $x$ is non-singular.  In Theorem \ref{t1} we give a complete classification in the case of the actions of $T_3(\omega)$ on $M_3(\omega)$.

Section \ref{sDD} is devoted to the study of the actions of the Drinfeld double $D(T_n(\omega))$ of $T_n(\omega)$ on matrix algebras. In Section \ref{ssADTN}, we give the matrices for the division action $D(T_n(\omega))$. The case of mixed action is treated, up to isomorphism, in the case of $D(T_2)$. In Section \ref{ssGEADD} we treat the case where the action of $D(T_n(\omega))$ is elementary.

In section \ref{suqsl2}, we describe the matrix form of the actions of $u_q(\Sl_2)$ on the matrix algebras. In Section \ref{ssBHA} we handle the inner actions of so called Book Hopf algebras from \cite{AS} and in Section \ref{ssC23} we classify the actions of $u_q(\Sl_2)$ on $M_2(\FF)$. In Section \ref{ssDDUQ} we note that all actions of $u_q(\Sl_2)$ can be lifted up to the elementary actions of  $D(T_q))$ but no division grading of $D(T_q)$ is a lifting of an action of $u_q(\Sl_2)$.

\section{Group gradings}\label{ssGG} Let $\cG=G(H)$ the group of group-like elements in $H$.  Denote by $\fG$ the group of multiplicative $\FF$-characters of $\cG$.  From now on we will be assuming that $\cG$ is abelian and $\FF$ is such that  $|\fG|=|\cG|$ (for instance, $\FF$ algebraically closed of characteristic coprime to $|\cG|$, or $\FF=\R$ and $|\cG|=2$, etc.). In this case, it is well-known that the algebra $\cA$ acquires a $\fG$-grading 
\begin{equation}\label{egr}
\Gamma: \cA=\oplus_{\vp\in\fG}\cA_\vp
\end{equation}
given by 
\[
\cA_\vp=\{ a\in \cA\;\vert\; g\ast a=\vp(g)a\mbox{ for any }g\in \cG\}.
\]
If $a\in \cA_\vp$ we also write $\deg a=\vp$. The set 
\[
\supp\Gamma=\{ \vp\in\fG\:|\:\cA_\vp\ne\{ 0\}
\]
is called the \textit{support} of the grading $\Gamma$. The subspaces $\cA_\gamma$ are called \textit{homogeneous components} of $\cA$. An element $a\in\cA_\gamma$ is called \textit{homogeneous of degree} $\gamma$. One writes $\deg a=\gamma$. A subspace $\cB$ is called \textit{graded} if $\cB=\bigoplus_{\gamma\in\fG}(\cB\cap\cA_\gamma)$.

Group gradings on $M_n(\FF)$ have been completely described over arbitrary algebraically closed fields $\FF$ (\cite{BZP}, see also \cite{NVO}, \cite[Chapter 2]{EK} and references therein). Let us call a graded algebra $\cA$ \textit{graded-simple} if it has no nontrivial proper graded ideals. A $\fG$-graded algebra $\cD$ is called \textit{graded-division} if all nonzero  homogeneous elements of $\cD$ are invertible. In full analogy with  the classical Wedderburn - Artin theorem, the following is true.

\begin{theorem}\label{t_2.6.}
Let $\fG$ be a group and let $\cA$ be a $\fG$-graded algebra over a field $\FF$. If $\cA$ is graded
simple satisfying the descending chain condition on graded left ideals, then there
exists a $\fG$-graded algebra $\cD$ and a $\fG$-graded right $\cD$-module $V$ such that $\cD$ is a $\fG$-graded
division algebra, $V$ is finite-dimensional over $\cD$, and $\cA$ is isomorphic to $\End_\cD V$
as a $\fG$-graded algebra.
\end{theorem} 

Note that if $\{ e_1,\ld,e_m\}$ is a graded basis of $V$ as a right vector space over $\cD$, with $\deg e_i=\gamma_i,\: i=1,\ld,n$, then $\End_\cD V$ is spanned by the elements $E_{ij}d$ where $d$ is a graded element of $\cD$ and $E_{ij}(e_j)=e_i$. The grading is given  by 
\[
\deg E_{ij}d=\gamma_i(\deg d)\gamma_j^{-1}.
\]
Using Kronecker products, we can write $\cA\cong M_n(\FF)\otimes \cD$ and identify $E_{ij}d$ with $E_{ij}\ot d$. Then $\deg (E_{ij}\ot d)=\gamma_i(\deg d)\gamma_j^{-1}$. 

If $\fG$ is abelian then we can simply say that $\deg E_{ij}d=(\deg E_{ij})(\deg d)$. We often say that $V$ is a right  graded vector space over a graded division algebra $\cD$. If $\cD=\FF$, the above grading is called \textit{elementary}, defined by the $n$-tuple $\gamma_1,\ld,\gamma_n$. If $\dim_\cD V=1$ then the grading is called \textit{division}. In this latter case, any nonzero $a\in \cA_\chi$ is invertible, for any $\chi\in G$.  We will use the terms \textit{elementary} and \textit{division} also for the actions of a Hopf algebra $H$, with $G(H)$ abelian, on a matrix algebra $M_m(\FF)$ in the case where the respective grading of $M_m(\FF)$ by $\fG$ is elementary or division. A simple remark is the following: 

\begin{remark}\label{r_no_division} If $|G(H)|$ is coprime with $m$, then the action of $H$ on $M_m(\FF)$ is elementary. Also if $0\ne p=\mathrm{char} (\FF)$ and $p|m$ then the action of $H$ on $M_m(\FF)$ is elementary.
\end{remark}

If the grading is elementary and $\dim_\FF V=n$ then the grading is completely defined by the \textit{dimension function} $\kappa: \Gamma\to\NN_0$ defined as follows. Let $V=\oplus_{\gamma\in \Gamma}V_\gamma$. Then $\kappa(\gamma)=\dim V_\gamma$.  The group $\fG$ naturally acts on the set of such functions if one defines $(\gamma\ast\kappa)(\gamma')=\kappa(\gamma\gamma')$. Two elementary gradings, defined by the dimension functions $\kappa$ and $\kappa'$, are isomorphic if and only if $\kappa'\in \fG\ast\kappa$.   Let us take a graded basis $\{ e_1,\ld,e_n\}$ in $V$, with $\deg e_i=\gamma_i$. Then $\kappa(\gamma)$  is the number of elements $e_i$ in the basis such that  $\deg e_i=\gamma$. A graded basis of $\End_\FF V$ is then formed by the maps $E_{ij}$, as above, and $\deg E_{ij}=\gamma_i\gamma_j^{-1}$.  

\medskip

No description of graded division algebras $\cD$ over arbitrary fields is available, but in the case of an algebraically closed field of characteristic zero or coprime to the orders of elements in $\fG$ they are known to be isomorphic to the twisted group algebras $\cD=\FF^\sigma\fT$, where $\sigma\in Z^2(\fT,\FF^\times)$ is a 2-cocycle (actually, bicharacter) on $\fT$ with values in $\FF^\times$. Two algebras corresponding to the cocycles $\sigma_1$ and $\sigma_2$ are graded isomorphic if and only if $\sigma_1$ and $\sigma_2$ belong to the same cohomology class in $Z^2(\fT,\FF^\times)$. In the cohomology class of a bicharacter $\sigma$ there is precisely one \textit{alternating bicharacter} $\beta: \fT\times\fT\to\FF^\times$, given by  
\[
\beta(\vp,\psi)=\frac{\sigma(\vp,\psi)}{\sigma(\psi,\vp)}.
\]
An algebra $\cD=\FF^\sigma\fT$ is simple if and only if $\beta: \fT\times\fT\to\FF^\times$ is nonsingular, that is, $\Ker\beta=\{\vp\:|\:\beta(\vp,\psi)=1, \mbox{ for all }\psi\in\fT\}$ is trivial. In the Classification Theorem \ref{t_2.6.}, in the case of simple algebras, the graded division algebra $\cD$ is simple, as an ungraded algebra. As a result, $\cD$ is  defined, up to isomorphism, by its support $\fT$ and a nonsingular alternating bicharacter $\beta:\fT\times\fT\to\FF^\times$. Different bicharacters lead to nonisomorphic graded division algebras $\cD(\fT,\beta)$. Note that each homogeneous component $\cD_\tau$ of $\cD(\fT,\beta)$ is one dimensional, and its basis $\{ X_\tau\}$, can be chosen so that 
\begin{equation}\label{e0101}
X_\tau^{o(\tau)}=I\mbox{ and }X_{\tau_2}X_{\tau_1}=\beta(\tau_2,\tau_1)X_{\tau_1}X_{\tau_2}\mbox{ where }o(\tau)\mbox{ is the order of }\tau\in\fT. 
\end{equation}
\begin{theorem}\label{t_2.15}\emph{(\cite{BSZ01}, also \cite[Chapter 2]{EK})} Let $\fT$ be a finite abelian group and let $\FF$ be an algebraically closed field. There exists a grading on the matrix algebra $M_n(\FF)$ with support $\fT$ making $M_n(\FF)$ a graded division algebra over $\fT$  if and only if $\charac\:\FF$ does not divide $n$ and
\begin{equation}\label{e444}
\fT\cong \ZZ_{\ell_1}^2\times\cdots\times \ZZ_{\ell_r}^2,\mbox{ where }\ell_1\cdots\ell_r=n.
\end{equation}
The isomorphism classes of such gradings
are in one-to-one correspondence with nondegenerate alternating bicharacters
$\beta : \fT\times \fT\to\FF^\times$. All such gradings belong to one equivalence class.
\end{theorem}

An explicit matrix realization of $\cD(\fT,\beta)$ is the following. First of all, using the above equation (\ref{e444})
\[
\cD(\fT,\beta)\cong\cD(\ZZ_{\ell_1}^2,\beta_1)\otimes\cD(\ZZ_{\ell_2}^2,\beta_2)\ot\cdots\ot\cD(\ZZ_{\ell_r}^2,\beta_r)
\]
where $\cD(\ZZ_n^2\,\beta)\cong \cA=M_n(\FF)$ and the grading on $M_n(\FF)$ is given, as follows. If $\ZZ_n^2\cong (\mu)\times(\nu)$, where $o(\mu)=o(\nu)=n$ and $\omega$ is a primitive $n$th root of 1, then one considers Sylvester's \textit{clock} and \textit{shift} matrices
\[
\cC=\!\!\begin{pmatrix}
1&0&0&\cdots&0&0\\
0&\omega&0&\cdots&0&0\\
0&0&\omega^2&\cdots&0&0\\
\cdots&\cdots&\cdots&\cdots&\cdots&\cdots\\
0&0&0&\cdots&0&\omega^{n-1}
\end{pmatrix}\!\!,\,
\cS=\!\!\begin{pmatrix}
0&1&0&\cdots&0&0\\
0&0&1&\cdots&0&0\\
\cdots&\cdots&\cdots&\cdots&\cdots&\cdots\\
0&0&0&\cdots&0&1\\
1&0&0&\cdots&0&0
\end{pmatrix}\!\!.
\]
Then the matrices $X_{\mu^k\nu^\ell}=\cC^k \cS^\ell$ satisfy the conditions of Equation (\ref{e0101}), with $\beta(\mu,nu)=\omega$, and the grading of  $M_n(\FF)$ is given by setting 
\[
\cA_{\mu^k\nu^\ell}=\Span\{X_{\mu^k\nu^\ell}\}=\Span\{X_{\mu^k}X_{\nu^\ell}\}=\Span\{X_{\mu}^k X_\nu^\ell\}, 
\mbox{ for all }0\le k,\ell<n.
\]

In conclusion, as we noted,  if $\cA$  is simple then $\cD$ is also a simple algebra. If the field $\FF$ is algebraically closed, $\cD\cong M_k(\FF)$. As a result, if $\cA$ is a matrix algebra over an algebraically closed field, then, for some $n,k$, we have $\cA\cong M_n(\FF)\otimes M_k(\FF)$ where the grading on $M_n(\FF)$ is elementary, while on $M_k(\FF)$ the grading is division.

\section{Actions on graded algebras}\label{ssAGrA}
In this paper we are interested in the actions of finite-dimensional pointed Hopf algebras with abelian group of group-likes on matrix algebras. We start with the actions of group algebras.

\subsection{Actions of group algebras}\label{ssAGA} Here we have: $H=\FF\cG$ is the group algebra of an abelian group $\cG$. Let  $\cA$ be a matrix algebra algebra and $u:H\to\cA$ a convolution invertible map making this action inner. In this case, the elements $u(g)$ can be explicitly written, as follows.

Let, as before, $\fG$ be the group of characters of $\cG=G(H)$. According to \cite{BK} (see also \cite[Corollary 2.43]{EK}, if we are given a  $\fG$-grading of $\cA=M_n(\FF)\cong\End_{\cD}(V_{\cD})$, then the matrix $u(g)$ of the action of for any $g\in\cG$, in an appropriate basis of $V_{\fD}$, can be written as follows. Let $\fT\subset \fG$ be the support of $\fD$ so that $\fD_\tau=\Span\,X_\tau$, for each $\tau \in\fT$. Suppose that $\sigma_1,\ld,\sigma_s$ are some elements of  $\fG$ which are pairwise non-congruent $\!\!\!\!\mod\fT$ and such that 
\[
V_\cD=(V_{\sigma_1}\ot\cD)\oplus\cdots\oplus (V_{\sigma_s}\ot\cD)
\]
is the decomposition of $V$ as a graded vector space over $\cD$ , and all elements of  $V_{\sigma_i}$ are of degree $\sigma_i$, $\dim V_{\sigma_i}=d_i$, $i=1,\ld,s$. Then the action of $g\in \cG$ can be given as the conjugation by the matrix
\begin{equation}\label{emixed-conjugation}
u(g)=(\sigma_1(g)I_{d_1}\ot X_{f(g)})\oplus\cdots\oplus(\sigma_s(h)I_{d_s}\ot X_{f(g)}),
\end{equation}
where $f(g)$ is a uniquely defined element of  $\cT$ such that 
\[
\tau(g)X_\tau=g\ast X_\tau=X_{f(g)}X_\tau X_{f(g)}^{-1}=\beta(f(g),\tau)X_\tau.
\]
We have $\beta(f(g),\tau)=\tau(g)$. Since $\beta$ is nonsingular, one easily checks that the map $f:g\mapsto f(g)$ is a well-defined homomorphism from $\cG$ to $\fT$.

The isomorphism classes of these actions are described by the same parameters as the isomorphism classes of corresponding gradings.

\subsection{Skew-primitive elements}\label{ssSPE} If $H$ is a pointed non-cosemisimple finite-dimensional Hopf algebra with an abelian group $\cG$ of group-likes then there are always elements $h,k\in\cG$, a character $\chi:\cG\to\FF^\times$ with $\chi(g)\ne 1$ and an element $x\not\in\FF\cG$ such that
\begin{gather}
gxg^{-1}=\chi(g)x\mbox{ for all }g\in\cG,\label{e000}\\
\Delta(x)=x\ot h+k\ot x.\label{e0000}
\end{gather}
For an $(h,k)$-primitive element $x$, one has 
\[
\ve(x)=0\mbox{ and }S(x)=-k^{-1}xh^{-1}. 
\]
If a Hopf algebra is generated by $\cG$ and $x$, as above, one can replace $x$ by a $(g,1)$ primitive element or $(1,g)$-primitive element, for a suitable $g\in \cG$.

Thus, when we study the actions of Hopf algebras, which are not group algebras, it is important to know the actions of the elements $x$, as above. We start with the effect of (\ref{e000}).

\begin{proposition}\label{p99} Let $\Gamma$ be a grading defined in \emph{(\ref{egr})}. Suppose $x$ is an element of $H$ such that there is a character $\chi: \cG\to \FF^\times$ satisfying $hx=\chi(h)xh$, for all $h\in\cG$. Let $\cK$ be the kernel of the action of $\cG$ on $\cA$, $\cK^\perp$ its orthogonal complement in $\fG$, $\fT$ the subgroup generated by $\supp\Gamma$ in $\fG$. Then the following are true.
\begin{enumerate}
\item[$\mathrm{(i)}$] The action of $x$ on $\cA$ maps any $\cA_\gamma$ to $\cA_{\chi\gamma}$.
\item[$\mathrm{(ii)}$]  If $\chi\not\in\cK^\perp$ then $x$ acts trivially on $\cA$.
\item[$\mathrm{(iii)}$]  If $\chi\not\in\fT$ then $x$ acts trivially on $\cA$.
\item[$\mathrm{(iv)}$]  If there is natural $m$ such that $(\supp\Gamma)^m=\{\ve\}$ and $\chi^m\ne \ve$ then $x$ acts trivially on $\cA$.
\end{enumerate}
\end{proposition}
\begin{proof}
\begin{enumerate}
\item[$\mathrm{(i)}$] Indeed, if $a\in \cA_\chi$ and $h\in\cG$ then 
\begin{eqnarray*}
 h\ast(x\ast a)&=&(hx)\ast a= \chi(h)(xh)\ast a= \chi(h)(x\ast (h\ast a)=\chi(h)\gamma(h)(x\ast a)\\&=&(\chi\gamma)(h)(x\ast a).
\end{eqnarray*}
It then follows that $x\ast a\in \cA_{\chi\gamma}$.
\item[$\mathrm{(ii)}$] Note that $\supp\Gamma\subset\cK^\perp$. Indeed, if $\vp\in\supp\Gamma$ then there is $0\ne a\in\cA_\vp$. Take any $h\in\cK$ then $\vp(h)a=ha=a$. It follows that $\vp(h)=1$, hence $\vp\in \cK^\perp$. 

To prove $xa=0$, for all $a\in\cA$, we write $a\in\cA$ is the sum of $a_\gamma\in\cA_\gamma$, $\gamma\in\supp\Gamma$. By the above, $\supp\Gamma\subset\cK^\perp$. If $\chi\not\in\cK^\perp$, then there is $h\in\cK$ such that $\chi(h)\ne 1$. We have $xa_\gamma\in\cA_{\chi\gamma}$, by $\mathrm{(i)}$. Now
$xa_gamma=h(xa)=(\chi\gamma)(h)(xa)$. But $(\chi\gamma)(h)=\chi(h)\gamma(h)=\chi(h)\ne 1$. Thus $xa=0$, as claimed.

\item[$\mathrm{(iii)}$] Follows easily from $\mathrm{(ii)}$, considering $\cK^\perp$ is a subgroup of $\fG$.

\item[$\mathrm{(iv)}$] In this case, also $\fT^m=\{ 1\}$ so that $\mathrm{(iii)}$ applies.

\end{enumerate}

\end{proof}

\subsection{Taft pairs and algebras}\label{ssTPA} 

 Now let $H$ contain a group-like element $g$ and a $(1,g)$-primitive element $x$ such that 
\begin{equation}
g^n=1, \: x^n=0,\: gx=\omega xg, \mbox{ where }\omega\mbox{ is a primitive }n{\mathrm{th}}\mbox{ root of }1.
\end{equation}  
We have
\begin{gather*}
\Delta(g) = g\ot g,\: \Delta(x)=x\ot 1+ g\ot x,\\ 
\ve(g)=1,\: \ve (x)=0, \:S(g)=g^{-1},\: S(x)=-g^{-1}x.
\end{gather*} 
One calls such a pair of elements a \textit{Taft $\omega$-pair}.  In this case, $g^sx=\chi(g^s)g^sx$, where $\chi(g^s)=\omega^s$. Let $\cG$ be the cyclic subgroup generated by $g$.

An $n^2$-dimensional Hopf algebra generated by an $\omega$-pair $g,x$ is called an $n^2$-dimensional \textit{Taft $\omega$-algebra}, denoted by $T_n(\omega)$, or simply $T_n$, if $\omega$ is fixed. 

Since in this case $\fG\cong\ZZ_n$, it follows from the above, that $\cA$ acquires a $\ZZ_n$-grading, as follows. Let $\chi\in\fG$ be given by $\chi(g)=\omega$. Set $\cA_k=\cA_{\chi^k}$. Then $\cA=\oplus_{k=0}^{n-1}\cA_k$ is the desired grading. 

\begin{corollary}\label{l0} 
If the action of $\cG$ is not faithful then the action of $x$ is trivial. If the support of the grading is contained in a proper subgroup $\fT$ of $\fG$ then the action of $x$ is trivial.
\end{corollary}
\begin{proof}
Let $e\ne g^k$, $k\,|\, n$ act trivially on $\cA$. Then the support of the grading is contained in proper subgroup $\fT=(g^k)^\perp$ of $\fG$, whereas $\chi$ generates $\fG$. So we can use Proposition \ref{p99} (ii).
\end{proof}


\section{Inner actions on central algebras}\label{sACSA} 

An algebra $\cA$ with identity element $1$ is called central if its center coincides with $\FF.1$. Suppose that a Hopf algebra $H$ acts on $\cA$, and this action is inner via the convolution invertible map $u:H\to\cA$. Then, for any group-like element $g\in\cG$ and any $(h,k)$-primitive element $x$, $h,k\in\cG$, one has
\begin{equation}\label{e0}
g\ast a=u(g)au(g)^{-1},\;x\ast a=u(x)au(h)^{-1}-u(k)au(k)^{-1}u(x)u(h)^{-1}.
\end{equation}
In the case where $x$ is $(1,k)$-primitive, we would have
\begin{equation}\label{e2}
g\ast a=u(g)au(g)^{-1},\;x\ast a=u(x)a-u(k)au(k)^{-1}u(x).
\end{equation}
In the case where $x$ is $(h,1)$-primitive, we would have
\begin{equation}\label{e1}
g\ast a=u(g)au(g)^{-1},\;x\ast a=u(x)au(h)^{-1}-au(x)u(h)^{-1}=[u(x),a]u(h)^{-1}.
\end{equation}
Here $[a,b]=ab-ba$ is the usual commutator of $a,b\in \cA$.
\begin{lemma}\label{lchange_u} If $g,x$ act as in \emph{(\ref{e2})}, then setting $u'(x)=u(x)-\lambda u(g)$, $u'(g)=\mu u(g)$ does not  change the action. In the case of \emph{(\ref{e1})}, the action is preserved if we replace $u(x)$ by $u(x)-\lambda.I$ and $u(g)$ by $\mu u(g)$. In both cases,  $\lambda\in\FF$ and $\mu\in\FF\setminus\{ 0\}$,
\end{lemma}
\begin{proof}
Indeed, if we denote the ``old'' action by $\ast$ and the ``new'' one by $\circ$, then 
\begin{eqnarray*}
g\circ a &=& u'(g)a u'(g)^{-1}=u(g)au(g)^{-1}=g\ast a\\
x\circ a&=&u'(x) a-u'(g) au'(g)^{-1}u'(x)\\&=&\left(u(x)-\lambda u(g)\right)a-u(g)au(g)^{-1}\left(u(x)-\lambda u(g)\right)\\&=&u(x)a-u(g)au(g)^{-1}u(x)=x\ast a,
\end{eqnarray*}
as claimed.

The case of (\ref{e1}) is even simpler.
\end{proof} 

A much more general isomorphism result about the actions of pointed Hopf algebras is the following.

\begin{theorem}\label{liso_inn} Let $H$ be a finite-dimensional pointed Hopf algebra with abelian group of group-likes, $\cA$ a matrix algebra. If two inner actions $\ast$ and $\circ$ defined by the functions $u:H\to \cA$ and $v:H\to \cA$ are isomorphic then there exists an invertible element $C\in \cA$ and $\lambda(g),\mu(g,x)\in\FF$,  such that $u(g)=\lambda(g) Cv(g)C^{-1}$, where $\lambda(g)\in\FF$, for any grouplike $g\in H$ and $u(x)=Cv(x)C^{-1}+\mu(g,x)u(g)$, $\mu(g,x)\in\FF$, for any $(1,g)$-primitive element $x\in H$. In the case of a $(g,1)$-primitive element $y$, we must have $u(y)=Cv(y)C^{-1}+\mu(g,y) I$.
\end{theorem}

\begin{proof}
If $g\in G(H)$, then $\vp:\cA\to \cA$ is an isomorphism of actions, then $g\ast \vp(A)=\vp(g\circ A)$, for any $A\in\cA$. Since $\cA$ is central simple, there is invertible $C\in \cA$ such that $\vp(A)=CAC^{-1}$. So we have $u(g)CAC^{-1}u(g)^{-1}=Cv(g)Av(g)^{-1}C^{-1}$, for all $A\in \cA$. As a result, $v(g)^{-1}C^{-1}u(g)C^{-1}$ is a central element of $\cA$. Since $\cA$ is central, $v(g)^{-1}C^{-1}u(g)C^{-1}=\lambda(g)\in\FF$ and $u(g)=\lambda(g) Cv(g)C^{-1}$, as claimed. We can also write  $u(g)C=\lambda(g)Cv(g)$ and $v(g)^{-1}C^{-1}=\lambda(g)C^{-1}u(g)^{-1}$.

Now let $x$ be a $(1,g)$-primitive element of $H$ and suppose $x\ast \vp(A)=\vp(x\circ A)$. Then
\[
u(x)CAC^{-1}-u(g)CAC^{-1}(g)^{-1}u(x)=C(v(x)A-v(g)Av(g)^{-1}v(x))C^{-1}.
\]
It follows that 
\begin{eqnarray*}
&&u(x)CAC^{-1}-Cv(x)AC^{-1}=u(g)CAuC^{-1}(g)^{-1}u(x)-v(g)Av(g)^{-1}v(x))C^{-1}\\
&&(u(x)C-Cv(x))A= Cv(g)A(\lambda C^{-1}u(g)^{-1}u(x)C-v(g)^{-1}v(x))\\
&&(v(g)^{-1}C^{-1}u(x)C-v(g)^{-1}v(x))A=A(\lambda C^{-1}u(g)^{-1}u(x)-v(g)^{-1}v(x))\\
&&(v(g)^{-1}C^{-1}u(x)C-v(g)^{-1}v(x))A=A(v(g)^{-1}C^{-1}u(x)C-v(g)^{-1}v(x)).
\end{eqnarray*}
As a result, $v(g)^{-1}C^{-1}u(x)C-v(g)^{-1}v(x)$ is a central element in $\cA$, and so there exists $\mu(g,x)\in\C$ such that $u(x)=Cv(x)C^{-1}+\mu(g)u(g)$, as claimed.
\end{proof}

Conversely, suppose our conditions hold and $H$ is generated by its group-likes and skew-primitives. Then 
\[
v(g)^{-1}C^{-1}u(g)C^{-1}\mbox{ and }v(g)^{-1}C^{-1}u(x)C-v(g)^{-1}v(x)
\]
are central elements. We can read our calculations in the reverse order, to find that $\vp(A)=CAC^{-1}$ preserves the action of group-like and skew-primitive elements. Since $H$ is generated by these elements, we find that $\vp$ is an isomorphism of actions $\ast$ and $\circ$.

\section{Pointed Hopf algebras}\label{PHARO}
Pointed Hopf algebras with abelian coradical have been classified thanks to efforts of many authors, specially N. Andruskiewitchsh, H.-J. Schneider, I. Heckenberger and  I. Angiono. A useful survey of the development in the area is \cite{AGI}. In this paper we will be looking at actions of a class oh pointed Hopf algebras in the paper \cite{AS}. 

Let $\cG$ be a finite nontrivial abelian group. A \textit{datum} for a\textit{ quantum linear space} $\cR=\cR(a_1, \ld, a_n;\chi_1, \ld , \chi_n)$ consists of the elements
$a_1, \ld, a_n\in \cG$, $\chi_1, \ld , \chi_n\in\fG=\wh{\cG}$, such that  the following
conditions are satisfied:
\begin{gather}
q_j=\chi_j(a_j)\ne 1\label{ejj}\\
\chi_j(a_i)\chi_i(a_j)=1\mbox{ for all }1\le i\ne j\le n.\label{eji}
\end{gather}
One says that the datum, or its associated quantum linear space, has rank $n$. If $N_j=o(q_j)$, the order of $q_i$ in $\cG$, then $\cR=\cR(a_1, \ld, a_n;\chi_1, \ld , \chi_n)$ is an algebra with the following presentation in terms of generators and defining relations:
\begin{equation}
\cR=\langle x_1,\ld,x_n\:|\: x_i^{N_i}=0;\:x_ix_j=\chi_j(a_i)x_jx_i\mbox{ for all }1\le i\ne j\le n\rangle.
\end{equation}

Let us fix a decomposition $\cG=(g_1)\times\cdots\times(g_k)$ and denote by $M_\ell$ the order of $g_\ell$, $1\le\ell\le k$. Let $a_1, \ld, a_n\in \cG$, $\chi_1, \ld , \chi_n\in\fG$ be the datum for a quantum linear space $\cR=\cR(a_1, \ld, a_n;\chi_1, \ld , \chi_n)$, that is, (\ref{ejj},\,\ref{eji}) hold. 

A compatible datum $D$ for $\cG$ and $\cR$ consists of a vector $(\mu_1,\ld,\mu_n)$, $\mu_i\in\{ 0,1\}$ and a matrix $(\chi_{ij})$, $\lambda_{ij}\in\FF$, $1\le i,j\le n$, such that
\begin{enumerate}
\item[(i)]  if $a_i^{N_i}=1$ or $\chi_i^{N_i}\ne \ve$, then $\mu_i=0$,
\item[(ii)]    if  $a_ia_j=1$ or $\chi_i\chi_j\ne \ve$, then $\lambda_{ij}= 0$.
\end{enumerate}

\begin{definition}\label{dPHA}
Let $\cG$ be a finite abelian group, $\cR$ a quantum linear
space, and $D$ a compatible datum.  Keep the notation above. Let $\cP(\cG, \cR, D)$ be the algebra presented by generators $g_\ell$,  $1\le\ell\le k$ , and $x_i$, $1\le i\le n$, with defining relations
\begin{gather}
g_\ell^{M_\ell}=1,\:1\le\ell\le k;\label{e5.3}\\
g_\ell g_t = g_t g_\ell,\: 1\le \ell, t\le k; \label{e5.4}\\
 x_ig_\ell=\chi_i(g_\ell) g_\ell x_i,\: ,\: 1\le\ell\le k; \label{e5.5}\\
x_i^{N_i}=\mu_i(1-a_i^{N_i}), 1\le i\le n; \label{e5.6}\\
x_jx_i =\chi_i(a_j)x_ix_j +\lambda_{ij}(1-a_ia_j), 1\le i,j\le n\label{e5.7}.
\end{gather}
A unique Hopf algebra structure on $\cP(\cG, \cR, D)$ is  given by
\begin{gather*}
\Delta(g_\ell)=g_\ell\otimes g_\ell,\: \Delta(x_i)=x_i\ot 1+a_i\ot x_i,\:1\le\ell\le k,\:1\le i,j\le n;\\
S(g_\ell)=g_\ell^{-1},\: S(x_i)=-a_i^{-1}x_i,\: 1\le\ell\le k,\:1\le i\le n.
\end{gather*}
\end{definition}

One can also say that $H$ is generated by its group $\cG=G(H)$ of group-likes, which is abelian, and the set of $(1,a_i)$-primitive $x_i$, for some element $a_i\in\cG$, $i=1,\ld,n$. These elements satisfy the above relations (\ref{e5.5}, \ref{e5.6}, \ref{e5.7}). In these relations, $\chi_1,\ld,\chi_n$ are multiplicative characters $\chi_i:\cG\to\FF^\times$, $N_i=o(\chi_i(a_i))\ne 1$. If one denotes $q_i=\chi_i(a_i)$, then $N_i$ is the order of $q_i$ in the group $\FF^\times$. The elements $\mu_i$ and $\lambda_{ij}$ are subject to the compatibility conditions (i), (ii), just before Definition \ref{dPHA}.  

Finally, the number $n$ of generators of the quantum vector space $\cR$ is called the \textit{rank} of the Hopf algebra $\cP(\cG,\cR,D)$.

\begin{remark}\label{rKern}
Let $\cP(\cG,\cR,D)$ act on an algebra $\cA$ so that, for some $i=1,\ld,n$, the action of $x_i$ is nonzero. If $g$ acts trivially on $\cA$ then $\chi_i(g)=1$. If $a_i$ acts trivially then $x_i$ acts trivially, too. 
\end{remark}

This follows by Proposition \ref{p99} and by (\ref{ejj}).

\subsection{Inner actions of pointed Hopf algebras $\cP(\cG, \cR, D)$}\label{ssIAHHARO} In this section we study the actions of pointed Hopf algebras of rank one on matrix algebras. We set $H=\cP(\cG, \cR, D)$, as above.

\begin{theorem}\label{pCR} Let $u:H\to \cA$ define an inner action of a Hopf algebra $H$ on a matrix algebra $\cA$. Then for each $i=1,\ld,k$, there is a function $\lambda_i: \cG\to \FF$ such that, for any $g\in \cG$ one has 
\begin{equation}\label{e0001}
u(g)u(x_i)u(g)^{-1}=\chi_i(g)u(x_i)+\lambda_i(g)u(a_i).
\end{equation}
One can choose $u$ so that $\lambda(g)=0$, that is, 
\begin{equation}
u(g)u(x_i)=\chi_i(g) u(x_i)u(g),\mbox{ for all }g\in\cG,
\end{equation}
the action remaining the same. 
\end{theorem}

\begin{proof}
Note that since $\cG$ is abelian, we have $a_i*(g* A)=g*(a_i*A)$, for any $g\in \cG$, $A\in \cA$. Thus 
\[
u(g)u(a_i)Xu(a_i)^{-1}u(g)^{-1}=u(a_i)u(g)Xu(g)^{-1}u(a_i)^{-1}.
\]
Now let us compute $g*(x_i*A)=\chi_i(g)x_i*(g*A)$ in terms of the inner action via the map $u:H\to \cA$. We have
\begin{eqnarray}
g*(x_i*A)&=&u(g)(u(x_i)A-u(a_i)Au(a_i)^{-1}u(x_i))u(g)^{-1}\nonumber\\
&=&u(g)u(x_i)Au(g)^{-1}-u(g)u(a_i)Au(a_i)^{-1}u(x_i)u(g)^{-1}\nonumber\\
&=&u(g)u(x_i)Au(g)^{-1}-u(a_i)u(g)Au(g)^{-1}u(a_i)^{-1}u(g)u(x_i)u(g)^{-1}\label{e001}
\end{eqnarray}
\begin{eqnarray}
\chi_i(g)x_i*(g*A)&=&\chi_i(g)x_i*(u(g)Au(g)^{-1})\nonumber\\
&=&\chi_i(g)u(x_i)u(g)Au(g)^{-1}-\chi_i(g)u(a_i)u(g)Au(g)^{-1}u(a)^{-1}u(x_i)\label{e002}
\end{eqnarray}
If we subtract  (\ref{e002}) from (\ref{e001}), we obtain
\begin{eqnarray*}
0&=&(u(g)u(x_i)-\chi_i(g)u(x_i)u(g))Au(g)^{-1}\\
&-&u(a_i)u(g)A(u(g)^{-1}u(a_i)^{-1}u(g)u(x_i)- \chi(g)u(g)^{-1}u(a_i)^{-1}u(x_i)u(g))u(g)^{-1}.
\end{eqnarray*}
Multiplying on the left by $u(g)^{-1}u(a_i)^{-1}$ and on the right by $u(g)$, we obtain
\begin{eqnarray*}
&&(u(g)^{-1}u(a_i)^{-1}u(g)u(x_i)-\chi_i(g)u(g)^{-1}u(a_i)^{-1}u(x_i)u(g))A\\
&=&A(u(g)^{-1}u(a_i)^{-1}u(g)u(x_i)-\chi(g)u(g)^{-1}u(a_i)^{-1}u(x_i)u(g)).
\end{eqnarray*}

Since $A$ is arbitrary and $\cA$ central simple there is $\lambda_i(g)\in\FF$ such that
\begin{eqnarray*}
u(g)^{-1}u(a_i)^{-1}u(g)u(x_i)&=&\chi(g)u(g)^{-1}u(a_i)^{-1}u(x_i)u(g))+\lambda_i(g) 1_\cA,\\
u(g)u(x_i)u(g)^{-1}&=&\chi_i(g)u(x_i)+\lambda_i(g)u(a_i), 
\end{eqnarray*}
proving (\ref{e0001}). 

If $g_1,g_2\in \cG$, then
\begin{gather}
u(g_2)u(x_i)-\chi_i(g_2)u(x_i)u(g_2)=\lambda_i(g_2)u(a_i)u(g_2)\label{e555}\\
u(g_1)u(g_2)u(x_i)-\chi_i(g_2)u(g_1)u(x_i)u(g_2)=\lambda_i(g_2)u(g_1)u(a_i)u(g_2)\nonumber
\end{gather}
Since $u(g_1g_2)$ is a scalar multiple of $u(g_1)u(g_2)$, we can write
\begin{gather*}
u(g_1g_2)u(x_i)-\chi_i(g_1g_2)u(x_i)u(g_1g_2)=\lambda_i(g_1g_2)u(a_i)u(g_1g_2)\\
u(g_1)u(g_2)u(x_i)-\chi_i(g_1g_2)u(x_i)u(g_1)u(g_2)=\lambda_i(g_2)u(a_i)u(g_1)u(g_2)
\end{gather*}
It follows that
\begin{gather}
\chi_i(g_2)u(g_1)u(x_i)u(g_2)+\lambda_i(g_2)u(g_1)u(a_i)u(g_2)\nonumber\\
=\chi_i(g_1)\chi_i(g_2)u(x_i)u(g_1)u(g_2)+\lambda_i(g_1g_2)u(a_i)u(g_1)u(g_2)\label{e5555}
\end{gather}
Cancelling out $u(g_2)$ and substituting (\ref{e555}), where $g_2$ is replaced by $g_1$, to (\ref{e5555}), we obtain
\begin{gather*}
\chi_i(g_2)\chi_i(g_1)u(x_i)u(g_1)+\chi_i(g_2)\lambda_i(g_1)u(a)u(g_1)+\lambda_i(g_2)u(g_1)u(a_i)\\
=\chi_i(g_1)\chi_i(g_2)u(x_i)u(g_1)+\lambda_i(g_1g_2)u(a)u(g_1).
\end{gather*}
It follows that
\[
\chi_i(g_2)\lambda_i(g_1)u(a_i)u(g_1)+\lambda_i(g_2)u(g_1)u(a_i)=\lambda_i(g_1g_2)u(a_i)u(g_1).
\]
If we set $u(g)u(a_i)=\xi(g,a_i)u(a_i)u(g)$, for $\xi(g,a_i)\in\FF$, then we get
\begin{gather*}
\lambda_i(g_1g_2)=\lambda_i(g_1)\chi_i(g_2)+\lambda_i(g_2)\xi(g_1,a_i)\\
\lambda(g_2g_1)=\lambda(g_2)\chi_i(g_1)+\lambda_i(g_1)\xi(g_2,a_i)\\
\lambda(g_1)(\xi(g_2,a_i)-\chi_i(g_2))=\lambda(g_2)(\xi(g_1,a_i)-\chi(g_1))
\end{gather*}
If $g_1=g$, $g_2=a_i$ then
\begin{equation}\label{e55555}
\lambda_i(g)(1-\chi_i(a_i))=\lambda(a_i)(\xi(g,a_i)-\chi_i(g)).
\end{equation}
Since $\chi_i(a_i)\ne 1$, it follows that once $\lambda_i(a)=0$, all $\lambda_i(g)=0$. Now if we replace $u(x_i)$ by $\bar{u}(x_i)=u(x_i)+\frac{\lambda(a_i)}{1-\chi_i(a_i)}$, then we will have $u(a_i)\bar{u}(x_i)=\chi_i(a_i)\bar{u}(x_i)u(a_i)$. Our calculation shows that then also $u(g)\bar{u}(x_i)-\chi_i(g)\bar{u}(x_i)u(g)$,  for all $g\in\cG$, as claimed.
\end{proof}

An important consequence of this theorem is the following.

\begin{corollary}\label{tchi}
Let a Hopf algebra $H$ with a group of group-likes $\cG$ and $\fG=\wh{\cG}$ act on a matrix algebra $\cA$. Let 
\[
\Gamma: \cA=\bigoplus_{\vp\in\fG}\cA_\vp
\]
be the associated $\fG$ grading of $\cA$. If $x$ is a $(1,a)$-primitive element in $H$ for some $a\in\cG$ and $\chi\in\fG$ is such that $gx=\chi(g)x$, for all $g\in\cG$. Then $u(x)\in\cA_\chi
$. 
\end{corollary}

Let $H=\cP(\cG, \cR, D)$, as above.

\begin{theorem}\label{lSC} Let $u:H\to \cA$ define an inner action of a Hopf algebra $H$ on a matrix algebra $\cA$. Then there is $\sigma_i\in\FF$ such that the following relation holds for the elements $u(a_i)$ and $u(x_i)$ of the algebra $\cA$:
\begin{equation}\label{euxn}
u(x_i)^{N_i}=\mu_i.1_\cA+\sigma_i u(a_i)^{N_i}.
\end{equation}
\end{theorem}

\begin{proof} This is an application of the $q$-binomial formula \cite[Lemma 3]{Rad}. Let $L_y,R_z:\cA\to \cA$ denote the left, respectively, right multiplications by $y$, respectively $z$, in the algebra $\cA$. Then we can rewrite the action of $x$ on any $A\in \cA$ in  Equation (\ref{e2}) as 
\[
x_i\ast A=(L_{u(x_i)}-L_{u(a_i)}R_{u(a_i)^{-1}}R_{u(x_i)})(A)
\]
hence 
\[
x_i^{N_i}\ast A=(L_{u(x)}-L_{u(a)}R_{\q}R_{u(x)})^{N_i}(A).
\]
 Setting $\vp=L_{u(x)}$, $\psi=-L_{u(a)}R_{u(a)^{-1}}R_{u(x)}$, we get 
\begin{eqnarray*}
&&\psi\vp=-L_{u(a_i)}R_{u(a_i)^{-1}}R_{u(x_i)}L_{u(x_i)}\\
&=&-L_{u(a_i)}L_{u(x_i)}R_{u(a_i)^{-1}}R_{u(x_i)}\\&&=-L_{u(a_i)u(x_i)}R_{u(a_i)^{-1}}R_{u(x_i)}=-q_i^{-1}L_{u(x_i)u(a_i)}R_{u(a_i)^{-1}}R_{u(x_i)}\\&&=-q_i^{-1}L_{u(x_i)}L_{u(a_i)}R_{u(a_i)^{-1}}R_{u(x_i)}=-q_i^{-1}\vp\psi.
\end{eqnarray*} 

Applying the $q$-binomial formula to $(\vp+\psi)^n$, we get 
\begin{eqnarray*}
x_i^{N_i}\ast A&=&(L_{{u(x_i)}^{N_i}}+(-1)^{N_i} (L_{u(a_i)}R_{u(a_i)^{-1}}R_{u(x_i)})^{N_i})A\\
&=&{u(x_i)}^{N_i}A+ (-1)^{N_i}{u(a_i)}^{N_i}R_{u(a_i)^{-1} u(x_i)}^{N_i}A\\
&=&{u(x_i)}^{N_i}A+ (-1)^{N_i}R_{(u(a_i)^{-1} {u(x_i)})^{N_i}}A.
\end{eqnarray*} 
Now
\[
 (u(a_i)^{-1} {u(x_i)})^{N_i}=q_i^{\frac{{N_i}({N_i}-1)}{2}}{u(x_i)}^{N_i}{u(a_i)}^{-{N_i}}
 \]
 So, taking into account that 
 \[
 (-1)^{N_i}q_i^{\frac{{N_i}({N_i}-1)}{2}}=-1,
 \]
 for any natural $N_i$, we get
\[
\mu_i((A-u(a_i)^{N_i} A u(a_i)^{-{N_i}})=\mu_i(1-a_i^{N_i})\ast A
={u(x_i)}^{N_i}A-u(a_i)^{N_i}Au(a_i)^{-{N_i}}{u(x)}^{N_i}.
\]
This can be rewritten as
\[
Au(a_i)^{-{N_i}}(-\mu_i.1_\cA+u(x_i)^{N_i})=u(a_i)^{-{N_i}}(-\mu_i.1_\cA+u(x_i)^{N_i})A
\]
Since $\cA$ is central simple and  $A\in \cA$ is arbitrary, there is $\sigma_i\in\FF$ such that
\[
u(a_i)^{-{N_i}}(-\mu_i.1_\cA+u(x_i)^{N_i})=\sigma_i.1_\cA.
\]
Multiplying both sides by $u(a_i)^{N_i}$ on the left, we obtain the desired equation (\ref{euxn}).
\end{proof}

We keep setting $H=\cP(\cG, \cR, D)$, as before.

\begin{theorem}\label{pij-ji}
The following relations hold for the elements of the inner action of $H=\cP(\cG,\cR,D)$ on a matrix algebra $\cA$, defined by the convolution invertible map $u:H\to\cA$ . For any $1\le i,j\le n$ there exist $\zeta_{ij}\in\FF$ such that
\begin{equation}\label{eij-ji}
u(x_j)u(x_i)-\chi_j(a_i)u(x_i)u(x_j)=\lambda_{ij}.1_\cA+\zeta_{ij}u(a_i)u(a_j).
\end{equation}
\end{theorem}
\begin{proof} 
For the proof, we are going to use the defining relation (\ref{e5.7}). We compare the action of $x_jx_i-\chi_i(a_j)x_ix_j$ with the action of $\lambda_{ij}(1-a_ia_j)$ on the elements of $\cA$.

Let $\cG$ be the group of group-likes of $H$, $\fG$ the dual group for $\cG$. Then $\cA$ is graded by $\fG$. Thus, it is sufficient to check (\ref{eij-ji}) on the homogeneous elements of the grading. Let $A\in\cA_\delta$, $\delta\in\fG$. Then $x_i\ast A\in \cA_{\chi_i\delta}$ . So

\begin{gather}
x_j\ast(x_i\ast A)=u(x_j)(x_i\ast A)-(\chi_i\delta)(a_j)(x_i\ast A)u(x_j)\nonumber\\
=u(x_j)\left[u(x_i)A-\delta(a_i)Au(x_i)\right]-(\chi_j\delta)(a_j)\left[u(x_i)A-\delta(a_i)Au(x_i)\right] u(x_j)\nonumber\\
=u(x_j)u(x_i)A-\delta(a_i)u(x_j)Au(x_i)-(\chi_i\delta)(a_j)u(x_i)Au(x_j)+\delta(a_j)\delta(a_i)Au(x_i)u(x_j).\label{eij1}
\end{gather}

Likewise,

\begin{gather}
x_i\ast(x_j\ast A)=u(x_i)(x_j\ast A)-(\chi_j\delta)(a_i)(x_j\ast A)u(x_i)\nonumber\\
=u(x_i)\left[u(x_j)A-\delta(a_j)Au(x_j)\right]-(\chi_j\delta)(a_i)\left[u(x_j)A-\delta(a_j)Au(x_j)\right] u(x_i)\nonumber\\
=u(x_i)u(x_j)A-\delta(a_j)u(x_i)Au(x_j)-(\chi_j\delta)(a_i)u(x_j)Au(x_i)+\delta(a_i)\delta(a_j)Au(x_j)u(x_i).\label{eji1}
\end{gather}

We then have

\begin{gather*}
\chi_i(a_j)(x_i\ast(x_j\ast A))=\chi_i(a_j)u(x_i)u(x_j)A\\-\chi_i(a_j)\delta(a_j)u(x_i)Au(x_j)-\chi_i(a_j)(\chi_j\delta)(a_i)u(x_j)Au(x_i)\\+\chi_i(a_j)\delta(a_i)\delta(a_j)Au(x_j)u(x_i).
\end{gather*}

Note $\chi_i(a_j)\chi_j(a_i)=1$.
As a result,
 
\begin{gather}
x_j\ast(x_i\ast A)-\chi_i(a_j)(x_i\ast(x_j\ast A))\nonumber\\
=[u(x_j)u(x_i)-\chi_i(a_j)u(x_i)u(x_j)]A\nonumber\\+[\delta(a_ia_j)Au(x_j)u(x_i)-\delta(a_ia_j)\chi_i(a_j)Au(x_i)u(x_j)].\label{e35}
\end{gather}

Now let us compute the action of $\lambda_{ij}(1-a_ia_j)$ on $A\in\cA_\delta$. We have

\begin{gather}
\lambda_{ij}(1-a_ia_j)\ast A=\lambda_{ij}(1-\delta(a_ia_j))A\nonumber
\\=\lambda_{ij}A-\lambda_{ij}u(a_i)u(a_j)Au(a_j)^{-1}(a_i)^{-1}.\label{e36}
\end{gather}

When we equate (\ref{e35}) and (\ref{e36}) and rearrange terms, we will have

\begin{gather*}
[u(x_j)u(x_i)-\chi_i(a_j)u(x_i)u(x_j)-\lambda_{ij}]A\\
=\delta(a_ia_j)A[u(x_j)u(x_j)-\chi_i(a_j)u(x_i)u(x_j)-\lambda_{ij}]\\
=u(a_i)u(g_j)Au(a_j)^{-1}u(a_i)^{-1}[u(x_j)u(x_i)-\chi_i(a_j)u(x_i)u(x_j)-\lambda_{ij}]
\end{gather*}

or

\begin{gather*}
u(a_j)^{-1}u(a_i)^{-1}[u(x_j)u(x_i)-\chi_i(a_j)u(x_i)u(x_j)-\lambda_{ij}]A\\
=Au(a_j)^{-1}u(a_i)^{-1}[u(x_j)u(x_i)-\chi_i(a_j)u(x_i)u(x_j)-\lambda_{ij}]
\end{gather*}

Thus the element
\[
u(a_j)^{-1}u(a_i)^{-1}[u(x_j)u(x_i)-\chi_i(a_j)u(x_i)u(x_j)-\lambda_{ij}]
\]
is central in $\cA$, so that there exists $\zeta_{ij}\in\FF$ such that
\[
u(a_j)^{-1}u(a_i)^{-1}[u(x_j)u(x_i)-\chi_i(a_j)u(x_i)u(x_j)-\lambda_{ij}]=\zeta_{ij}.1_\cA,
\]
or
\[
u(x_j)u(x_i)-\chi_i(a_j)u(x_i)u(x_j)=\lambda_{ij}.1_\cA+\zeta_{ij}u(a_i)u(a_j).
\]
\end{proof}

Theorems \ref{pCR}, \ref{lSC} and \ref{pij-ji} work for Taft algebras and their doubles, but also Lusztig's kernel and so on. Not only do they give  the necessary but also they give sufficient conditions for the actions of $H$ on matrix algebras. Although it seems impossible to classify all sets of matrices satisfying all the relations given in these theorems, some interesting examples can still be provided. This is done in the remainder of the paper.

\subsection{Matrix form for the actions of  pointed Hopf algebras of rank 1, with cyclic group of group-likes}\label{ssMFCGGL}

Let $\cG$ be a cyclic group generated by an element $g$ of order $m$, with the group of characters $\fG$. Then there is $\gamma\in\fG$ such that $\gamma(g)=\omega$, a primitive $m$th root of 1. Let $\cR$ be a quantum torus of rank 1, with the compatible datum $D=a,\chi,x$, where $a\in\cG$, $\chi\in\fG$ such that $q=\chi(a)\ne 1$, $n=o(q)$. In the Hopf algebra $H=\cP(\cG, \cR, D)$ we have generators $g,x$ such that $gx=\chi(g)xg$ and $x^n=\mu(1-a^n)$. If $a^n=1$ or $\chi^n\ne\ve$, we must have $\mu=0$. Otherwise, any Hopf algebra is isomorphic to the one with either $\mu=0$ or $\mu=1$. The algebras with $\mu=0$ are sometimes called of nilpotent type.

Let us consider the action of such $H$ on a matrix algebra $\cA$. Then $\cA=\End\,V$, where $V$ is a finite-dimensional $\fG$-graded vector space over $\FF$: $V=\bigoplus_{\vp\in\fG}V_\vp$. If $f\in\cA_\tau$ then $f\ast\cA_\vp\subset\cA_{\tau\vp}$. As a result, $x\ast\cA_\vp\subset\cA_{\chi\vp}$. We can always choose $u(g)$ so that $u(g)^m=I$. It was shown in Theorem \ref{pCR} that $u(x)$ can be chosen so that $u(g)u(x)u(g)^{-1}=\chi(g)u(x)$. So, $u(x)\in\cA_\chi$.

There is a primitive $m$th root of 1 $\zeta\in\FF$ and a vector space decomposition $V=\sum_{s=0}^{m-1} V_s$ such that $u(g)|_{V_s}=\zeta^s\: \id_{V_{s}}$. Let $\ell$ be the order of $\chi$ and  $m=\ell r$. We rearrange the subspaces $V_s$  so that 
\begin{equation}\label{e9876}
V=\sum_{t=0}^{r-1} W_t\mbox{ where }W_t=\sum_{k=0}^{\ell-1} V_{t+kr}.
\end{equation}
Then each $W_t$ is invariant with respect to $u(x)$, which follows from $u(g)u(x)=\zeta^ru(x)u(g)$. In an appropriate basis, the restriction $u(x)^{(t)}$ of $u(x)$ to each $W_t$ has the form
\begin{equation}\label{e3791}
u(x)^{(t)}=\left(\begin{array}{ccccc} 
0&0&\ldots&0&u_{1\ell}^{(t)}\\
u_{21}^{(t)}&0&\ldots&0&0\\
\ld&\ld&\ld&\ld&\ld\\
0&0&\ld&u_{\ell 1}^{(t)}&0
\end{array}\right).
\end{equation}
where each $u_{i+1,i}^{(t)}$ is a rectangular $d_{k+1}^{(t)}\times d_k^{(t)}$ matrix of rank $r_{i}^{(t)}$. Here $\dim\, V_{t+kr}=d_k^{(t)}$ and $r_{i}^{(t)}=\rank\,u_{i+1,i}^{(t)}$.

Further restrictions on $u(x)$ appear when one uses Theorem \ref{lSC} and compares $u(x)^n$ with $\mu.\id_V+\lambda u(g)^n$. 

First assume that one of the conditions $a^n=1$ or $\chi^n\ne\ve$ holds. Then $H$ is nilpotent. We have $u(x)^n=\lambda u(a)^n$. Now $n$ is a divisor of $\ell$ and if $n\ne\ell$ (which is the same as $\chi^n\ne\ve$), we must have simply $u(x)^n=0$, because $u(a)^n$ is diagonal and $u(x)^n$  does not have nonzero diagonal blocks. Also, if $\chi^n=\ve$ (hence $\ell=n$) and $a^n=1$ then $u(a)^n$ is a scalar matrix and hence $u(x)^n$ is a scalar matrix. So in this case, $u(x)^n=\lambda.I$ is a scalar matrix. It is hard to say much if $\lambda=0$. But if $\lambda \ne 0$ then $u(x)$ is a nonsingular map.

One easily concludes the all the maps $u_{21}, u_{32},\ld,u_{1\ell}$,  are isomorphisms.  Indeed, the entries of $(u(x)^{(t)})^n$ are all cyclic permutations of   $u_{21} u_{32}\cdots u_{1\ell}$. Since $(u(x)^{(t)})^n$ is a nonzero scalar matrix, we have $\rank (u(x)^{(t)})^n=d_1^{(t)}+\cdots d_\ell^{(t)}$. If $d_k^{(t)}$ is the smallest of the sizes of these blocks, then the rank of each diagonal entry is at most this number. It follows that all $V_{t+kr}$ have the same dimension and all the blocks are nonsingular. By an appropriate choice of the bases in these subspaces we may assume that $u(x)^{(t)}$ has the form

\begin{equation}\label{e7491}
u(x)^{(t)}=\left(\begin{array}{ccccc} 
0&0&\ldots&0&\lambda I_{d_t}\\
I_{d_t}&0&\ldots&0&0\\
\ld&\ld&\ld&\ld&\ld\\
0&0&\ld&I_{d_t}&0
\end{array}\right).
\end{equation} 
 
If we assume that both $a^n\ne 1$ and $\chi^n=\ve$, then we either have $x^n=0$ or $x^n=1-a^n$. In the first case, $u(x)^n=\lambda u(a)^n$, in the second $u(x)^n=I+\lambda u(a)^n$, for some $\lambda\in\FF$. 

In both cases, on each $W_t$, the action of  $\mu.\id_V+\lambda u(a)^n$ is scalar with the coefficient $\mu_t=\mu+\lambda\zeta^{tn}$. At the same time, $(u(x)^{(t)})^n$ is a block-diagonal matrix, with blocks of the same size as those for $u(a)$. If $\mu_t=0$, then $(u(x)^{(t)})^n=0$ and we cannot say much more. Otherwise, the same argument as the one leading to $(\ref{e7491})$ works and we have that

\begin{equation}\label{e6491}
u(x)^{(t)}=\left(\begin{array}{ccccc} 
0&0&\ldots&0&\mu_t I_{d_t}\\
I_{d_t}&0&\ldots&0&0\\
\ld&\ld&\ld&\ld&\ld\\
0&0&\ld&I_{d_t}&0
\end{array}\right).
\end{equation} 

The results obtained above can be summarized as follows.

\begin{theorem}\label{t11}
Let a Hopf algebra $H=\cP(\cG, \cR, D) $ with $\cG=(g)$,  $\cR=\FF[x]$, be defined by $g^m=1$, $gx=\chi(g)xg$ and $x^n=\mu(1-a^n)$. Suppose $H$ acts on a matrix algebra $\cA=\End\,V$. Then the action is inner, via $u:H\to\cA$. There is an $m$th primitive root $\zeta$ of $1$ such that $V=\sum_{i=0}^{m-1}V_i$, where $V_i$ is an eigenspace for $u(a)$, with eigenvalue $\zeta^i$. If $m=nr$ then $V=\sum_{t=0}^rW_t$, each $W_t$ being invariant under both $u(g)$ and $u(x)$. In an appropriate basis of $W_t$, the matrix of $u(x)$ has the form of (\ref{e3791}), while the matrix of $u(a)$ is scalar with coefficient $\zeta^{t+kr}$. If $H$ is nilpotent, then either $u(x)^n=0$ or $u(x)$ takes the form (\ref{e7491}). If $H$ is not nilpotent and $u(x)$ acts on $W_t$ in a nonsingular way, then, the matrix $u(x)^{(t)}$ of this restriction takes the form of (\ref{e6491}).
\end{theorem}

\begin{remark}\label{r0} \emph{Note that, up to some change of generators, which we discuss at the beginning of Section \ref{sDD}, this applies to all, but one, types of pointed Hopf algebras of dimension $p^3$ from the classification in \cite{AS}. The exceptional case (a) from that paper is discussed in Section \ref{ssEP3A}. The actions in the case of the ``popular'' Taft algebras \cite{G} are discussed in detail in Section \ref{TOM}.}
\end{remark}

\subsection{Actions of $H$ and division gradings}\label{ssAHDDG}
Let $H$ act on a matrix algebra $\cA$, $\cG=G(H)$, $\chi\in\fG=\wh{\cG}$, $a\in\cG$, $x$ a $(1,a)$-primitive element satisfying $gx=\chi(g)x$ and $q=\chi(a)$ is of order $n\ne 1$. Suppose the action of $\cG$ on $\cA$ makes $\cA$ a graded division algebra. Let $\fT\subset\fG$ be the support of the grading. Since the grading is division, $\fT$ is a subgroup, whose order is a square $m^2$, where $m^2$, hence $m$ are the divisors of the order of $\cG$. The grading is accompanied by an alternating bicharacter $\beta:\fT\times\fT\to\FF^\times$. The basis of $\cA$ is formed by the elements $X_\vp$, $\vp\in\fT$, $X_\vp^{o(\vp)}=I$, and $X_\vp X_\psi=\beta(\vp,\psi)X_\psi X_\vp$, for all $\vp,\psi\in\fT$. Also, $o(\vp)$ is the order of $\vp$ in the group $\fG$. 

As mentioned earlier, the action of $g\in\cG$ is conjugation by a matrix $u(g)=X_{f(g)}$ such that $\beta(f(g),\vp)=\vp(g)$, for any $\vp\in\fT$. From Proposition \ref{pCR} it follows that $u(g)u(x)u(g)^{-1}=\chi(g)u(x)+\lambda(g)u(a)$. We know that $u'(x)=u(x)-\frac{\lambda}{1-\chi(a)}u(a)$ satisfies $u(g)u'(x)u(g)^{-1}=\chi(g)u'(x)$. If we write $u'(x)=\sum_{\vp\in\fT}\alpha_\vp X_\vp$ then

\begin{gather*}
\sum_{\vp\in\fT}\alpha_\vp X_{f(g)}X_\vp=\sum_{\vp\in\fT}\chi(g)\alpha_\vp X_\vp X_{f(g)}\\
=\sum_{\vp\in\fT}\alpha_\vp\beta(f(g),\vp)X_\vp X_{f(g)}=\sum_{\vp\in\fT}\chi(g)\alpha_\vp X_\vp X_{f(g)}.
\end{gather*}

Comparing like terms on both sides of the above equation, and considering $\beta(f(g),\vp)=\vp(g)$, we find $\alpha_\vp\vp(g)=\chi(g)\alpha_\vp$. Thus $\alpha_\vp\ne 0$ implies $\vp(g)=\chi(g)$, for all $g$. Hence $u'(x)=\alpha X_\chi$. As a result, 
\begin{equation}\label{euxGD}
u(x)=\alpha X_\chi+\frac{\lambda}{1-\chi(a)}X_{f(a)}.
\end{equation}
The value of the coefficient of $X_{f(a)}$ is unimportant, so we simply write 
\[
u(x)=\alpha X_\chi+\gamma X_{f(a)}.
\]
Now we need to check that $x^n$ and $\mu(1-a^n)$ act on $\cA$ in the same way. The necessary and sufficient conditions for this are given in Proposition \ref{lSC}: $\mu\cdot I+u(x)^n=\sigma u(a)^n$. Note that if $a^n=e$ or $\chi^n\ne\ve$ then $\mu = 0$. Otherwise, $\mu$ can be arbitrary, but actually, one can only consider the cases $\mu=0$ or $\mu=1$. Let us start with the case $\mu=0$. In this case, one should have
\[
(\alpha X_\chi+\gamma X_{f(a)})^n=\sigma X_{f(a)}^n
\]
Using the $q$-binomial formula, we get
\[
\alpha^n X_\chi^n+\gamma^n X_{f(a)}^n=\sigma^n X_{f(a)}^n.
\]

This shows that there was no need to view $u(x)$ in its general form $\alpha X_\chi+\gamma X_{f(a)}$ rather than simply $u(x)=\alpha X_\chi$. So we can continue as follows
\[
\alpha^nX_\chi^n=\sigma^nX_{f(a)}^n.
\]
So $\chi^n=f(a^n)$ is a necessary condition in this case. Since $\sigma$ is arbitrary, this condition is also sufficient.

Now let $\mu=1$, so that $a^n\ne e$ and $\chi^n=\ve$. In this case, we have 
\[
I+\alpha^n X_{\chi^n}=\sigma^nX_{f(a)}^n.
\]
Hence,
\[
(1+\alpha^n) I=\sigma^n X_{f(a)}^n.
\]
So we see that this can happen only when $f(a^n)=\ve$, which is the same as $a^n$ acts trivially on $\cA$. It follows that if such a Hopf algebra acts on a graded division algebra, this action cannot be faithful, even if we consider only group-like elements. 

\begin{theorem}\label{pR1HHADG}
Let $H= \cP(\cG,\cR,D)$ be a pointed Hopf algebra of rank 1, with admissible data $D=(a, \chi,\mu$, $o(\chi(a))=n$. Let $\fG$ be the group of characters on $\cG$. Assume that the action of $\cG$ on a matrix algebra $\cA$ makes $\cA$ a $\fG$-graded division algebra, with support $\fT$. In any extension of  this action of $\cG$ to $H$ we must have $u(x)=\alpha X_\chi$, for some $\alpha\in\FF$. If $\chi\not\in\fT$ then  the action of $x$ is zero. If $\chi\in\fT$ and $\chi^n=f(a)^n$ then extensions  always exist, with any $\alpha\in\FF$.  If $\mu=0$, this condition is also necessary. If $\mu=1$, in which case $\chi^n=\ve$ and $a^n\ne e$, the extensions exist even if $a^n$ acts non-trivially. In this case, $\alpha^n=-1$. 

 The actions via $u(x)=\alpha X_\chi$ and $ u^\prime(x)=\alpha^\prime X_\chi$  are isomorphic if and only if $\alpha=\alpha^\prime$. 
\end{theorem}

\subsubsection{An example of a $p^3$-algebra}\label{ssEP3A} In this section we consider the division actions of an algebra, which is Example (a) in the classification of pointed Hopf algebras of dimension $p^3$ in \cite{AS}. Given an odd prime number $p>2$, this is an algebra of the form $T_p(\omega)\ot \FF\ZZ_p$. In other words, this is a pointed Hopf Algebra $H$, of the type described in Section \ref{PHARO}, where $\cG=(g)\times (h)$ is an elementary abelian $p$-group, $\chi: \cG\to\FF^\times$ given by $\chi(g)=\omega$, $\chi(h)=1$, $\mu=0$. Let us consider the action of $H$ on a matrix algebra $\cA$ where the action of $\cG=G(H)$  induces a grading by $\fG$, which is a division grading. We have $\fG=(\mu)\times(\nu)$, where $\mu(g)=\omega$, $\mu(h)=1$, $\nu(g)=1$, $\nu(h)=\omega$. The grading is completely defined by an alternating bicharacter $\beta:\fG\times\fG\to\FF^\times$, given by $\beta(\mu,\nu)=\tau$ where $\tau$ is a $p$th primitive root of 1. In this case,  $\cA=M_p(\FF)$, is generated by the clock and shift matrices $X_\mu$ and $X_\nu$ and, as mentioned earlier, we may assume
 
\begin{gather*}
X_\mu=\diag(1,\tau,\tau^2,\ld,\tau^{p-1})\\
X_\nu=\begin{pmatrix}
0&1&0&\cdots&0&0\\
0&0&1&\cdots&0&0\\
\cdots&\cdots&\cdots&\cdots&\cdots&\cdots\\
0&0&0&\cdots&0&1\\
1&0&0&\cdots&0&0
\end{pmatrix};
\end{gather*}

These matrices satisfy $X_\mu^p=X_\nu^p=I$ and $X_\nu X_\mu=\beta(\mu,\nu)X_\mu X_\nu$.

Using Proposition \ref{pCR}, we can say that the degree of $u(x)$ is $\chi$, and so there is scalar $\alpha\in\FF$ such that $u(x)=\alpha X_\chi$. Now $\chi=\mu$ and so $u(x)=\alpha X_\mu$. As for $u(g)$ and $u(h)$, these can be found when we consider that the actions of $g$ and $h$ are conjugation by the matrices $X_{f(g)}$ and $X_{f(h}$, given by
 
\begin{gather*}
x\ast X_\vp=X_{f(g)}X_\vp X_{f(g)}^{-1}=\beta(\vp,f(g))X_\vp,\\
y\ast X_\vp=X_{f(h)}X_\vp X_{f(h)}^{-1}=\beta(\vp,f(h))X_\vp
\end{gather*}

So we must have $\beta(\vp,f(g))=\vp(g)$, for all $\vp\in\fG$. Similarly, $\beta(\vp,f(h))=\vp(h)$, for all $\vp\in\fG$. Taking $\vp=\mu,\:\nu$ and $f(g)=\mu^k\nu^\ell$, $f(h)=\mu^r\nu^s$, we easily find that $f(g)=\nu^{\ell}$ where $\tau^\ell=\omega$ and $f(h)=\mu^{-r}$, where $\tau^r=\omega^{-1}$. As a result, we may choose $u(g)=X_\nu^\ell$ and $u(h)=X_\mu^{-\ell}$, where $\tau^\ell=\omega$.

\begin{proposition}\label{pp3} Let $H$ be a pointed Hopf algebra of the form $T_p(\omega)\ot \FF\ZZ_p$. If $H$ acts on a matrix algebra $\cA$ so that the induced grading by $G(H)$ is a division grading, then $\cA\cong M_p(\FF)$ and the action is isomorphic to an inner action via a convolution invertible map $u: H\to\cA$ such that for some $1\le \ell<p$ and $\alpha\in\FF$, the matrices for the inner action are given by 
\[
u(g)=X_\nu^\ell,\: u(h)=X_\mu^{-\ell},\mbox{ where }\tau^\ell=\omega,\mbox{ and }u(x)=\alpha X_\mu.
\]
For different pairs $(\ell,\alpha)$ the actions are not isomorphic.

\end{proposition}

\subsection{Actions of $H=\cP(\cG, \cR, D)$ and mixed gradings}\label{ssAHDMG}

Let us use the notation and some facts from Section \ref{ssAGrA}.  Suppose a Hopf algebra $H=\cP(\cG, \cR, D)$ acts on a matrix algebra $\cA$. In this case, $\cA$ becomes a $\fG$-graded algebra. In this section we will look at the action of one of the $x_i$, which we denote by $x$. We also set $a_i=a$ and $\chi_i=\chi$. Indices $i,j$ will not be used in the sense of Definition \ref{dPHA}.

We set $d=d_1+\cdots+d_s$ and view our algebra $\cA$ as a Kronecker product $\cA=M_d(\FF)\ot\cD$. The action of $g$ is conjugation by $u(g)$, given in (\ref{emixed-conjugation}. To determine $u(x)$, we view this matrix as block-diagonal,  according to the above splitting of $d$, with coefficients in $\cD$. Then we can write 
\begin{equation}\label{eUXM}
u(x)=\sum_{\tau\in\cT} u^\tau\ot X_\tau\mbox{ and }u^\tau=\sum_{1\le i,j\le s}u_{ij}^\tau,\:u_{ij}^\tau\mbox{ being an }(i,j)\mbox{th block of  }u_{ij}^\tau.
\end{equation}

As we know from Theorem \ref{tchi}, for any $g\in\cG$, we must have $u(g)u(x)u(g)^{-1}=\chi(g)u(x)$, where $\chi$ is the selected character of $\cG$. We have
\[
u(g)u(x)u(g)^{-1}=\sum \sigma_i(g)u_{ij}^\tau\sigma_j^{-1}(g)\ot\tau(g)X_\tau=\chi(g)\sum u_{ij}^\tau\ot X_\tau.
\]
Since $g$ is arbitrary, the only terms $u_{ij}^\tau$ of $u(x)$ that survive in (\ref{eUXM}) are those, for which $\sigma_i\sigma_j^{-1}\tau=\chi$, or $\sigma_i=\sigma_j\chi\tau^{-1}$. 

We can formulate the result so far obtained, as follows.

\begin{proposition}\label{pmixed} Let $H$, as above, act on a matrix algebra $\cA$ so that the associated $\fG$-grading of $\cA$ is mixed. We keep the notation used in this section.  Then the inner action of $H$ on $\cA$ can be so chosen that
\begin{eqnarray*}
u(g)&=&(\sigma_1(g)I_{d_1}\ot X_{f(g)})\oplus\cdots\oplus(\sigma_s(g)I_{d_s}\ot X_{f(g)}),\;\forall g\in \cG,\\
u(x)&=&\sum_{1\le i,j\le s}u_{ij}\ot X_{\tau_{ij}},\mbox{ if } \tau_{ij}=\sigma_i^{-1}\sigma_j\chi\in\fT.
\end{eqnarray*}
In each block row (or column) of  $u(x)$ there is at most one block different from zero.
\end{proposition}

\begin{proof}
Only the last claim needs proof. But if we have $\sigma_i^{-1}\sigma_j\chi\in\fT$ and $\sigma_i^{-1}\sigma_k\chi\in\fT$, then $\sigma_j^{-1}\sigma_k\in\cT$, which is impossible.
\end{proof}
 
 In some cases more can be said.

{\bf Case 1.} An interesting particular case is where $\chi\in\fT$. In this case, only the elements $u_{ii}^\tau$ can be nonzero. Also, $\tau=\chi$. As a result, we must have
\[
u(x)=u_{11}\ot X_\chi\oplus\cdots\oplus u_{ss}\ot X_\chi, \mbox{ or }u(x)=u'(x)\ot X_\chi,
\]
where the grading of $u'(x)\ot I$ equals $\ve$.

{\sc Subcase} (a). Let us first assume that the datum $\mu$ from Definition \ref{dPHA} equals 0, that is, $x^n=0$. Then, as we know, we must have $u(x)^n=\alpha I$, for some $\alpha$.  If $\alpha\ne 0$, then we must also have $\chi^n=\ve$, and then $u_{ii}^n=\alpha I_{d_i}$, for all $i=1,\ld,s$. Thus, we may assume that in this case $u(x)=u'(x)\ot X_\chi$, where the $n$th power of $u'(x)$ is the scalar matrix. If $\chi^n\ne\ve$, this case is not possible. If $\alpha=0$ then $u'(x)^n=0$.

{\sc Subcase} (b). Let us consider the case where $\mu\ne 0$. Then we have $x^n=1- a^n$. We know that in this case, we must have $a^n\ne e$ and $\chi^n=\ve$. Actually, by Theorem \ref{lSC}, $ I +u(x)^n=\lambda u(a)^n$, for some $\lambda\in\FF$. Since $\chi^n=\ve$, we have
\begin{gather*}
u_{11}^n\ot X_\ve\oplus\cdots\oplus u_{ss}^n\ot  X_\ve+I_d\ot X_\ve\\=\lambda\left(\sigma_1(a^n) I_{d_1}\ot X_{f(a^n)}\op\cdots\op \sigma_s(a^n) I_{d_s}\ot X_{f(a^n)}\right).
\end{gather*} 
If $f(a^n)=\ve$, then, for any $\tau\in\fT$, we must have $\beta(f(a^n,\tau)=1$ or that $X_{f(a^n)}$ commutes with any $X_\tau$, $\tau\in\fT$. Thus $X_{f(a^n)}=X_\ve$. In this case, we must have
\[
u_{ii}^n=-1+\lambda\sigma_i(a)^n I_{d_i}, \mbox{ for all } 1\le i\le s.
\]
Thus the $n$th power of each $u_{ii}$ is a scalar matrix, but this time, if $\lambda\sigma_i(a)^n=1$, the matrix $u_{ii}$ is nilpotent, otherwise it is semisimple.

If $f(a^n)\ne\ve$ then we must have $\lambda=0$ and $\diag(u_{11},\ld,u_{ss})^n=-I$. This case  is quite similar to {\sc Subcase} (a), where $\mu=0$, except that $u'(x)$ is a nonsingular semisimple matrix with eigenvalues being $n$th roots of $-1$.

{\bf Case 2}. Generally, assume that $k>1$ is the order of $\chi$ in $\fG/\fT$.  Remember $V$ is a right graded vector space over a graded division algebra $\cD$ with the support $\fT$. We know that $u(x)$ can be so chosen that $u(x)$ maps $V_\vp$ to $V_{\vp\chi}$, for all $\vp\in\fG$.

We know that $V=\bigoplus_{\vp\in\fG} V_\vp$. Let $W_\bvp=\bigoplus_{\psi\in\vp\fT}V_\psi$. Then $V=\bigoplus_{\bvp\in\fG/\fT}W_\bvp$.

Let us denote by $k$ the order of $\chi$ in $\fG/\fT$. Assume $k>1$. Choose $\vp\in\Supp\Gamma$. The sum $U_\bvp=W_\bvp\bigoplus\cdots\bigoplus W_{\bvp\chi^{k-1}}$ is direct. We have $u(x):W_{\bvp\chi^{i}}\to W_{\bvp\chi^{i+1}}$, for $i=0,\ld,k-1$ and $u(x):W_{\bvp\chi^{k-1}}\to W_{\bvp}$.

By an appropriate choice of bases over $\cD$, we can ensure that the matrix for $P=u(x)$ on $U_\bvp$ is $k\times k$ block-diagonal, with only nonzero blocks $P_{i+1,i}$, which are matrices with coefficients in $\cD$. Moreover,  the blocks $P_{2,1}, P_{3,2},\ld,P_{k, k-1}$ can be chosen to have all their entries in the base field $\FF$. The last matrix $P_{k,1}$ is the matrix of the map $u(x):W_{\bvp\chi^{k-1}}\to W_{\bvp}$. If $\{ e_1^{(k-1)},...,e_{d_{k-1}}^{(k-1)}\}$ is a basis in $W_{\bvp\chi^{-k-1}}$, consisting of the elements of degree $\vp\chi^{-k+1}$, and $\{ e_1^{0},...,e_{d_{0}}^{(0)}\}$, consisting of the elements of degree $\vp$, then because, say, $\deg(u(x)(e_1^{(k-1)})=\vp\chi^k$, and $u(x)(e_1^{(k-1)})$ is a linear combination of homogeneous elements $ e_1^{0},...,e_{d_{0}}^{(0)}$ of degree $\vp$ with coefficients in $\cD$, all these coefficients must be homogeneous of degree $\chi^k\in\fT$, hence equal to the multiples of $X_{\chi^k}$. Thus, we have that all $P_{i+1,i}$ are of the forms $Y_{i+1}\ot I$ and $P_{1k}$ must have the form of $Y_1\ot X_{\chi^k}$.

Next we have to take $\overline{\psi}\not\in \{\bvp,\bvp\chi^{-1},\ld,\bvp\chi^{-k+1}\}$ and repeat the process. Finally, the matrix of $u(x)$ will be the sum of the blocks, corresponding to $U_\bvp$, where $\bvp$ runs through all representatives of the cosets of the subgroup $\fS$ of $\fG$ generated by $\chi$ and $\fT$. 

Now if we raise  the block corresponding to $U_\bvp$, to the power of $k$, we will obtain $\diag\{Z_1\ot X_{\chi^k},\ld,Z_k\ot X_{\chi^k}\}$, where $Z_i =Y_iY_{i+1}\cdots Y_{i-1}$, for all $i=1,2,\ld,k$.

As earlier, we first assume that the datum $\mu$ from Definition \ref{dPHA} equals 0, hence, $x^n=0$. Then, as we know, we must have $u(x)^n=\alpha I$, for some $\alpha$.  If $\alpha\ne 0$, then we must also have $\chi^n=\ve$, and then $Z_i^{n/k}=\alpha I_{d_i}$, for all $i=1,\ld,s$. Thus all $Z_i$ are non-singular. Also, in this case, as in Section \ref{ssMFCGGL}, we may assume $Y_2=\cdots=Y_k=I$ while $Y_1=Z\ot X_{\chi^k}$. As a result, in this case, we have 
\[
u(x)^{k}=\diag\{ \underbrace{Z\ot X_{\chi^k},\ld,Z\ot X_{\chi^k}}_k\}.
\]
Raising to the power $n/k$, we find that $Z^{n/k}=\alpha I$.  As a result, in this case ($\mu=0$ and $u(x)^n=\alpha I$, $\alpha\ne 0$), if $t=|\fG/\fS|$, the $u(x)$ can be reduced to the sum of blocks of the form
\begin{equation}\label{e7491G}
\left(\begin{array}{ccccc} 
0&0&\ldots&0&Z \ot X_{\chi^k}\\
I_{d_t}\ot X_\ve&0&\ldots&0&0\\
\ld&\ld&\ld&\ld&\ld\\
0&0&\ld&I_{d_t}\ot X_\ve&0
\end{array}\right).
\end{equation}
Here $Z$ is a matrix, individual for each coset of $\fG/\fS$, such that $Z^{n/k}=\alpha I$.

The determination of the canonical form for $u(x)$ in the cases $\mu=0$ and $u(x)n=0$, as well as in the case $\mu=1$ is even more technical, so we leave this for a future research.

\section{Action of $T_n(\omega)$ on $M_m(\FF)$}\label{TOM}

The goal of this section is the classification of the actions of $T_n(\omega)$ on $M_3(\FF)$.  We will apply the results of Section \ref{ssMFCGGL}. Possible forms for $u(x)$ are (\ref{e7491}) and (\ref{e3791}). But we give a more direct treatment here.

So we have $H=T_n(\omega)$ and $\cA=M_m(\FF)=\End V$, where $\omega$ is an $n$th primitive root of 1 and $\dim V=n$. Since $\fG$ is cyclic of order $n$,generated by an element $g$, the grading by $\fG$ on $\cA$ is elementary, that is, induced from the grading of $V=\sum_{\vp\in\fG}V_\vp$. Now we know that $u(x)$ can be chosen so that $u(x)\in \cA_\chi$ where $\chi(g)=\omega$. Also $u(x)^n=\alpha I$, for some $\alpha\in\FF$. This says that either $u(x)$ is semisimple nonsingular or $u(x)$ is nilpotent. Now, $u(x)(V_\vp)\subset V_{\vp\chi}$. So if there is $\vp\in\Supp V$ such that $\vp\chi\not\in \Supp V$ then $u(x)(V_\vp)=\{ 0\}$. In this case, $u(x)$ is a nilpotent. Otherwise, $u(x)$ is semisimple. As a result, we have two cases, as follows.

Case I: $V=\oplus_{i=0}^{n-1} V_{\chi^i}$ or

Case II: There are $\vp_1,\ld,\vp_t\in\fG$ such that $V$ splits as the direct sum of both $g$- and $x$-invariant nonzero subspaces 
\[
V^{(i)}=V_{\vp_i}\oplus V_{\vp_i\chi}\oplus\cdots\oplus V_{\vp_i\chi^{k-1}},
\]
such that $V_{\vp_i\chi^{k}}=\{ 0\}$.

In Case I, we the matrices for $u(g)$ and $u(x)$ have the following form:
\[
u(g)=\begin{pmatrix}
\vp(g)I_{d_0}&0&...&0\\
0&(\vp\chi)(g)I_{d_1}&...&0\\
...&...&...&...\\0&0&...&(\vp\chi^{n-1})(g)I_{d_{n-1}}
\end{pmatrix}
\]
where $\vp$ is any element of $\fG$.
\begin{equation}\label{e2191}
u(x)=\begin{pmatrix}
0&0&0&...&0&u_{0,n-1}\\
u_{10}&0&0&...&0&0\\
0&u_{21}&0&...&0&0\\
...&...&...&..&...&...\\
0&0&0&...&u_{n-1,n-2}&0
\end{pmatrix}
\end{equation}
We know that $u(x)^n=\alpha I$, for some $\alpha\in\FF$. If $\alpha\ne 0$, then arguing as in Section \ref{ssMFCGGL}, we will obtain
\[
u(x)=\begin{pmatrix}
0&0&...&0&\alpha I_d\\
I_d&0&...&0&0\\
0&I_d&...&0&0\\
...&...&..&...&...\\
0&0&...&I_d&0
\end{pmatrix}
\]
where $d$ is the common dimension of all $V_\vp$. 

We can state it, as follows.

\begin{theorem}\label{tATA}
Let $u:T_n(\omega)\to M_m(\FF)$ define an inner action $T_n(\omega)$ on $M_m(\FF)$such that $u(x)$ is nonsingular. Then  $d=\frac{n}{m}$ is an integer and there exists $\alpha\in\FF^\times$ such that the action is isomorphic to the one where
\[
u(g)=\begin{pmatrix}
I_{d}&0&...&0\\
0&\omega I_{d}&...&0\\
...&...&...&...\\0&0&...&\omega^{n-1} I_{d}
\end{pmatrix}\!\!,\;
u(x)=\begin{pmatrix}
0&0&...&0&\alpha I_d\\
I_d&0&...&0&0\\
0&I_d&...&0&0\\
...&...&..&...&...\\
0&0&...&I_d&0
\end{pmatrix}\!\!.
\]

Two actions with different values of $\alpha$ are not isomorphic.$\hfill\Box$
\end{theorem}

Otherwise, if $\alpha=0$, we can conjugate $u(x)$ in the form (\ref{e2191}) by the block diagonal matrix $T=\diag\{ T_1,\ld,T_n\}$ then  each block $u_{j+1,j}$ will be replaced by $T_{j+1}u_{j+1,j}T_{j}^{-1}$. This allows one to assume we can assume that in the matrix of $u(x)^{(i)}$ all blocks $u_{j+1,j}$ have the form where only the first $r_{j+1}$ rows are different from zero, where $r_{j+1}$ is the rank of  $u_{j+1,j}$.

In Case II, the matrices of $u(g)$ and $u(x)$ split into the blocks determined by $V=\oplus_{i=1}^tV^{(i)}$. We will write $u(g)^{(i)}$ and $u(x)^{(i)}$ for the $i$th blocks of $u(g)$ and $u(x)$. Then
\[
u(g)^{(i)}=\begin{pmatrix}
\vp_i(g)I_{d_0}&0&...&0\\
0&(\vp_i\chi)(g)I_{d_1}&...&0\\
...&...&...&...\\0&0&...&(\vp_i\chi^{k-1})(g)I_{d_{k-1}}
\end{pmatrix}
\]
Also,
\[
u(x)^{(i)}=\begin{pmatrix}
0&0&0&...&0&0\\
u_{10}&0&0&...&0&0\\
0&u_{21}&0&...&0&0\\
...&...&...&..&...&...\\
0&0&0&...&u_{k-1,k-2}&0
\end{pmatrix}
\]
Here $d_0,d_1,\ld,d_{k-1}$ are the dimensions of  $V_{\vp_i},V_{\vp_i\chi},\ld,V_{\vp_i\chi^{k-1}}$. If we conjugate by the block diagonal matrix $T=\diag\{ T_0,\ld,T_{k-1}\}$ then $u(g)^{(i)}$ does not change while in $u(x)^{(i)}$ each block $u_{j,j-1}$ will be replaced by $T_ju_{j,j-1}T_{j-1}^{-1}$.

\subsubsection{Actions of $T_n(\omega)$ on $M_3(\FF)$}\label{ssT3M3}

Let $\cA=M_3(\FF)$, $H=T_n(\omega)$, $n\ge 3$, $\omega$ a primitive $n$th root of 1. The action of $g$ via matrix $u(g)=Q$ makes $\cA$ a $\ZZ_n$-graded algebra, $\cA_i=\{a\;|\; QaQ^{-1}=\omega^i a\}$. We also set $P=u(x)$ and choose $P$ is such a way that $QP=\omega PQ$. As a result, $P$ is of degree $1$.

It easily follows from the isomorphism criterion Proposition \ref{liso_inn} that we have only three nontrivial options for $Q$:

\begin{enumerate}
\item[\sc{Case 1}] $Q(k)=\diag(1,\omega^k I_2)$, where $0<k<n$,
\item[\sc{Case 2}] $Q(p,q)=\diag(1,\omega^p,\omega^q)$, where $0<p<q<n$.
\end{enumerate}

In {\sc Case 1},  the trivial component of the grading is the sum of two matrix subalgebras $e_1Ae_1\oplus e_2Ae_2$, where $e_1=E_{11}$ and $e_2= E_{22}+E_{33}$.  Other components of the grading are $\cA_{-k}=e_1Ae_2$ and $\cA_k=e_2Ae_1$.  So if $k\ne\pm 1\mod n$ then $P=0$. If $k=1$ then $P=e_2 Pe_1$, if $k=n-1$, then $P=e_1 Pe_2$. In both cases, the action will not change, if we conjugate both $Q$ and $P$ by a matrix $T$ of the form $T=\diag(a,D)$, where $a\in\FF$ and $D\in M_2(\FF)$.  Computing $P'=TPT^{-1}$, we obtain $P'=e_1(aPD^{-1})e_2$, in the first case, and $P'=e_2(aDP)e_1$  in the second. Clearly, we can choose $T$ so that $P$ gets the form $P(1)=E_{31}$, and in the second, $P(2)=E_{13}$. 

{\sc Case 2}. This time, the components of the grading are $\cA_0=\Span\{ E_{11}, E_{22}, E_{33}\}$,. Other components are $\cA_p=\Span\{ E_{21}\}$, $\cA_{-p}=\Span\{ E_{12}\}$,  $\cA_q=\Span\{ E_{31}\}$, $\cA_{-q}=\Span\{ E_{13}\}$,  $\cA_{p-q}=\Span\{ E_{23}\}$, $\cA_{q-p}=\Span\{ E_{32}\}$. If none of the numbers $\pm p$, $\pm q$, $\pm (p-q)$ is congruent 1 $\mod n$, we have to set $P=0$. 

We can write $P$ as follows: $P=\alpha E_{12}+\beta E_{23}+\gamma E_{31}$. We can conjugate $P$ by any nonsingular diagonal matrix $T=\diag(a,b,c)$. We have $P'=TPT^{-1}=(a/b)\alpha E_{12}+(b/c)\beta E_{23}+(c/a)\gamma E_{31}$. If one or two of the numbers $\alpha,\beta,\gamma$ equal zero then the remaining numbers can be made equal 1, which gives 6 pairwise nonisomorphic options for $P$:
\begin{eqnarray*}
&&P(3)_1=E_{21},\,P(3)_2=E_{32},\,P(3)_3=E_{13},\,P(3)_4=E_{21}+E_{32},\\
&&P(3)_5=E_{32}+E_{13},\, P(3)_6=E_{13}+E_{21}.
\end{eqnarray*}
In the case where all three of the numbers $\alpha,\beta,\gamma$ are nonzero, which is only possible if $n\le 3$, the best we can do is to reduce $P$ to $P(3)_\gamma=E_{21}+E_{32}+\gamma E_{13}$ where $\gamma$ is a nonzero number. Since $P(3)^3=\gamma I$, for different values of $\gamma$ the actions are not equivalent.

\begin{theorem}\label{t1} If $n\ne 3$ then any action of $T_n(\omega)$ on $M_3(\FF)$ is isomorphic to one of the inner actions via $u: H\to M_3(\FF)$ where the pair $(u(g),u(x)) $ is one of the following $(Q(1),P(1))$, $(Q(n-1),P(2))$, $(Q(1,n-1),P(3)_i)$, $i=1,\ld 6$. If $n=3$,  then we can also have $(Q(1,2),P(3)_\gamma)$, $\gamma\in\FF^\times$. All these actions are pairwise non-isomorphic.
\end{theorem}

\section{Actions of $D(T_n(\omega))$ on $M_m(\FF)$}\label{sDD}

The Drinfeld double of a Taft algebra is a particular case of algebras $\cP(\cG,\cR,D)$ defined in Section \ref{PHARO}. We will use the following presentation of the Drinfeld double $H=D(T_n)$ of  the Taft algebra $T_n(\omega)$, in the case where $n$ is an odd number $>1$. As  an algebra, $H$ is generated by $g, x, G, X$, where $g,G$ are group-likes and $x,X$ are skew-primitives  and

\begin{gather*}
g^n=G^n=1,\: x^n=X^n=0,\\ 
gG=Gg,\: gx=\omega^{-1} xg,\:  Gx=\omega^{-1} xG,\: gX=\omega Xg,\: GX=\omega XG,\\ 
\label{e_double}  xX-\omega Xx=1-gG\\ 
\Delta(x)=x\ot 1+g\ot x,\; \Delta(X)=X\ot 1+G\ot X.
\end{gather*}

Generally, it is possible in the definition of $\cP(\cG,\cR,D)$ to replace some of the variables $x_i$ to $x_ia_i^{-1}$ and $a_i$ to $a^{-1}$. In that case, the relations (\ref{e5.7}) between an ``old'' and ``new'' $x_i$, $x_j$ will change to $x_jx_i-x_ix_j=g-G$. The relations between two ``old'' and two `new'' variables will looks the same way. In particular, this can be done with the generators of the Drinfeld double to get $xX-Xx=G-g$. In the next result we  will show the connection between elements of the inner action when this change of ``old'' to ``new'' takes place.

\begin{proposition}\label{p1aa1} Let $u:H\to\cA$ provides an inner action of a Hopf algebra $H$ on a central algebra $\cA$. Suppose that $\Delta (x)=x\ot a+1\ot x$ is a $(a,1)$-primitive element of $H$, $a$ a grouplike. Let $x_1=xa^{-1}$. Then $\Delta (x_1)=x_1\ot a^{-1}+a^{-1}\ot x_1$ and 
\begin{equation}\label{e1aa1}
u(x)=u(x_1)u(a)+\lambda I\mbox{ and } u(x_1)=u(x)u(a)^{-1}-\lambda u(a),\mbox{ for some }\lambda\in\FF.
\end{equation} 

\end{proposition}
\begin{proof}
Let $A$ be an arbitrary element of $\cA$; we compute the action of $x_1$ on $A$ in two different ways, using formulas (\ref{e0},\ref{e2},\ref{e1}). Then
\[
x\ast A=u(x)Au(a)^{-1}-Au(x)u(a)^{-1},\; x_1\ast A=u(x_1)A-u(a)^{-1}Au(a)u(x_1).
\]
Since $x=x_1a$, we have
\[
x\ast A=x_1\ast(u(a)Au(a)^{-1})=u(x_1)u(a)Au(a)^{-1}-u(a)^{-1}u(a)Au(a)^{-1}u(a)u(x_1)
\]
As a result, we have
\[
u(x)Au(a)^{-1}-Au(x)u(a)^{-1}=u(x_1)u(a)Au(a)^{-1}-Au(a)^{-1}u(a)u(x_1).
\]
Then $(u(x)-u(x_1)u(a))Au(a)^{-1}= A(u(x)u(a)^{-1}-u(a)^{-1}u(a)u(x_1))$ or 
\[
(u(x)-u(x_1)u(a))A=A(u(x)-u(x_1)u(a)).
\]
Since $A$ is central, there is $\lambda\in \FF$, such that 
\[
u(x)-u(x_1)u(a)=\lambda I.
\]
From this equation both equations (\ref{e1aa1}) are immediate.
\end{proof}

The above proposition is important because in distinction from Definition \ref{dPHA}, where all skew primitive elements are $(1,a)$-primitive, in some ``real life'' example, like algebras of dimension $p^3$ in \cite{AS}, some skew-primitive elements are $(1,a)$-primitive and some are $(b,1)$-primitive. The above proposition allows one to switch in our formulas for the inner actions from one type of skew primitive elements to the other, without repeating similar arguments. Note that the additional summand $\lambda u(a)^{-1}$ and $\lambda I$ in (\ref{e1aa1}) can be ignored when computing the action of $x$ and $y$. This fact was proved in Lemma \ref{lchange_u}.

\subsection{Inner actions of $D(T_n)$}\label{ssADTN}

If $\cA$ is an $H$-algebra then we have $g\ast 1=G\ast 1=1$ and $x\ast 1=X\ast 1=0$.
Finally, for any $a,b\in \cA$, we have
\begin{eqnarray*}
g\ast(ab)&=&(g\ast a)(g\ast b),\; G\ast(ab)=(G\ast a)(G\ast b),\\
x\ast(ab)&=&(x\ast a)b-(g\ast a)(x\ast b),\; X\ast(ab)=(X\ast a)b-(G\ast a)(X\ast b).
\end{eqnarray*}

If $\cA$ is a matrix algebra, there is a map $u:H\to \cA$ such that, for any $a\in \cA$, we have

\begin{gather*}
h\ast a=u(h)au(h)^{-1},\mbox{ for any }h\in \cG=G(H),\\
x\ast a=u(x)a-u(g)au(g)^{-1}u(x),\: X\ast a=u(X)a-u(G)au(G)^{-1}u(X).
\end{gather*}

In the case, where the group $\cG=G(H)$ of grouplikes of a Hopf algebra $H$ acting on a matrix algebra $\cA=M_m(\FF)$ is not cyclic we need to consider the gradings on the matrix algebras which do not need to be elementary. 

Now $\cG=G(H)=(g)_n\times(G)_n\cong\ZZ_n\times\ZZ_n$. The dual group $\fG$ consists of pairs $(\chi,\psi)$ of characters $\chi\in\wh{(g)}$, $\psi\in\wh{(G)}$. Each such pair can be identified with the pair $(k,\ell)$ of numbers modulo $n$. We have $a\in \cA_{k,\ell}$ if and only if $g\ast a=\omega^k a$, $G\ast a=\omega^\ell a$. It follows from the defining relations that $x\ast \cA_{k,\ell}\subset \cA_{k-1,l+1}$ and $X\ast \cA_{k,\ell}\subset \cA_{k+1,l-1}$. For example, if the action of $(g)$ is not faithful then the action of both $x$ and $X$ is trivial. The proof of Lemma \ref{l0} works both for $x$ and $X$. Similarly, with $G$ in place of $g$. At the same time, it is possible that the action of $gG$ is trivial and $x$, $X$ still act in a non-trivial way. Actually, in this case one speaks about the actions of $H=u_q(\Sl_2)$ (see Section \ref{suqsl2}). Note that in this case, according to Proposition \ref{p99}, the support of the grading is a subgroup of $\fG$ equal to the annihilator of $gG$, which is a cyclic subgroup isomorphic to $\ZZ_n$. So the gradings lifted from $H=u_q(\Sl_2)$ must be elementary.


The following proposition is a direct consequence of Theorems \ref{pCR}, \ref{lSC}, \ref{pij-ji}, 

\begin{proposition}\label{p_main_for_DD} If $H=D(T_n)$ acts on a matrix algebra $\cA$ then the action is inner and the matrices of the inner action can be chosen so that
\begin{equation}\label{eugux}
u(g)u(x)u(g)^{-1}=\omega^{-1} u(x),\: u(G)u(x)u(G)^{-1}=\omega^{-1} u(x),
\end{equation}
\begin{equation}\label{euGuX}
u(g)u(X)u(g)^{-1}=\omega u(X),\: u(G)u(X)u(G)^{-1}=\omega u(X).
\end{equation}
and 
\begin{equation}\label{e_main for_DD_ne}
 u(x)u(X)-\omega u(X)u(x)=I+\lambda u(g)u(G)\mbox{ for some constant }\lambda\in\FF.
\end{equation}
If the accompanying grading by $\fG$ is not elementary then
\begin{equation}\label{e_main for_DD}
 u(x)u(X)-\omega u(X)u(x)=I.
\end{equation}
\end{proposition}

\begin{proof}
It follows from (\ref{e_main for_DD}) that if the grading is not elementary then all matrices on the left hand side, as well as $I$ have degree $\ve$, while $u(g)u(G)$ has degree different from $\ve$ (see (\ref{emixed-conjugation}). As a result, in this case, we must have a more restrictive relation (\ref{e_main for_DD}).
\end{proof}

Let us define a character $\chi:\cG\to\FF^\times$ by setting $\chi(g)=\omega^{-1}$, $\chi(G)=\omega$.

The sufficient conditions for the action of $D(T_n)$ on $M_m(\FF)$ will now look like

\begin{gather*}
u(g)^n= I_m, u(G)^n=I_m,\\
 u(h)u(x)=\chi(h) u(x)u(h),\: u(h)u(X)=\chi^{-1}(h) u(X)u(h),\mbox{ for any }h\in\cG\\
 u(x)^n=\zeta I,\: u(x)^n=\xi I,\: \zeta,\xi\in\FF\\
 u(x)u(X)- \omega u(X)u(x)=I+\lambda u(g)u(G), \: \lambda\in\FF.
 \end{gather*}
 
 An explicit form for of the division actions of $D(T_n(\omega))$ can be given, as follows. First, if $x$ and $X$ act nontrivially, then by Theorem \ref{p99}, the support of the grading cannot be a proper subgroup of $\fG$. So if $\fG=(\mu)\times(\nu)$, $o(\mu)=o(\nu)=n$, then also $\fT=(\mu)\times(\nu)$. We also need to fix an alternating bicharacter $\beta:\fT\times\fT\to\FF^\times$. This is done by fixing a primitive $n$th root of 1, say $\pi$. Now, as we concluded above, $u(x)=\gamma X_\chi^{-1}$ while $u(X)=\beta X_\chi$, for some $\gamma,\delta\in\FF$.  It follows from (\ref{e_main for_DD}) that 
 \[
 u(x)u(X)- \omega u(X)u(x)=\gamma\delta(1-\omega).I=I.
 \]
Thus we must have $\gamma\delta=\dfrac{1}{1-\omega}$.

Now to determine $u(g)$ and $u(G)$, we should act exactly as in Section \ref{ssEP3A}. Easy calculations show that we must have $u(g)$ and $u(G)$ be scalar multiples of
\[
u(g)=X_\nu^{-\ell}\mbox{ and }u(G)=X_\mu^{\ell}\mbox{ where }\ell\mbox{ is uniqely defined by }\pi^\ell=\omega.
\]
One easily checks, for example,

\begin{gather*}
u(G)u(x)u(G)^{-1}=(X_\mu^{\ell^{-1}})(\gamma X_\chi^{-1})(X_\mu^\ell)=\beta(\mu^{-\ell},\chi^{-1})X_\chi^{-1}\\
=\beta(\mu^{\ell},\mu^{-1}\nu^{-1})X_\chi^{-1}=\beta(\mu,\nu)^{-\ell}X_\chi^{-1}=\pi^{-\ell} X_\chi^{-1}=\omega^{-1} u(g).
\end{gather*}

An explicit example for $n=3$ looks like the following. Fix a 3rd primitive root $\omega$ of 1 and consider two options. If $\pi=\omega^2$ then

\begin{gather*}
u(g)=\begin{pmatrix}0&1&0\\0&0&1\\1&0&0\end{pmatrix},\:u(G)=\begin{pmatrix}1&0&0\\0&\omega^2&0\\0&0&\omega\end{pmatrix},\\u(x)=\gamma\begin{pmatrix}0&1&0\\0&0&\omega^2\\\omega&0&0\end{pmatrix},\: u(X)=\delta \begin{pmatrix}0&0&\omega^2\\1&0&0\\0&\omega&0\end{pmatrix}.
\end{gather*}

If $\pi=\omega$ then

\begin{gather*}
u(g)=\begin{pmatrix}0&0&1\\1&0&0\\0&1&0\end{pmatrix},\:u(G)=\begin{pmatrix}1&0&0\\0&\omega&0\\0&0&\omega^2\end{pmatrix},\\u(x)=\gamma\begin{pmatrix}0&1&0\\0&0&\omega\\\omega^2&0&0\end{pmatrix},\: u(X)=\delta \begin{pmatrix}0&0&\omega\\1&0&0\\0&\omega^2&0\end{pmatrix}.
\end{gather*}

We also must have $\gamma\delta=\dfrac{1}{1-\omega}$.

\subsection{Mixed actions of $D(T_2)$}\label{ssGEADD2}

In this section we will use the following notation for the Sylvester/Pauli matrices in the case $n=2$: 
\[
A=\left(\begin{array}{cc}1&0\\0&-1\end{array}\right),\:B=\left(\begin{array}{cc}0&1\\1&0\end{array}\right),\:C=\left(\begin{array}{cc}0&1\\-1&0\end{array}\right).
\]
Then, if $D(T_2)$ acts on $R=M_{2k}(\FF)$ such that the grading by $\fG$ is mixed, the matrices of the inner action can be chosen as follows:
\[
u(g)=I_k\ot A,\: u(G)=I_k\ot B.
\]
The commutation relations (\ref{eugux}) and  (\ref{euGuX}) will now take the form

\begin{gather*}
u(g)u(x)=-u(x)u(g),\:u(G)u(x)=-u(x)u(G),\\u(g)u(X)=-u(X)u(g),\:u(G)u(X)=-u(X)u(G).
\end{gather*}

It is now easy to compute the matrices $u(x)$ and $u(X)$:
\begin{equation}\label{euxuX}
u(x)=P\ot C,\: u(X)=Q\ot C\mbox{ for some }P,Q\in M_k(\FF).
\end{equation}
Using (\ref{e_main for_DD}), we get the following
\[
I_{2k}=u(x)u(X)+u(X)u(x)=(P\ot C)(Q\ot C)+Q\ot C)(P\ot C)=-(PQ+QP)\ot I_2.
\]
\begin{equation}\label{e888}
 PQ+QP=-I_k.
\end{equation} 
Note that if $T\in \GL_{k}$ then one can conjugate $R$ by $T\ot I_2$ without changing the grading. This allows one to reduce, say $P$ to a canonical form $P_1=TPT^{-1}$ and find $Q_1= TQT^{-1}$ from the equation 
$P_1Q_1+Q_1P_1=-I_k$.

One also needs to remember that $(P_1\ot C)^2=-P_1^2\ot I_2$ and $(Q_1\ot C)^2=-Q_1^2\ot I_2$ are scalar matrices. So in our future calculations we may assume that $P$ and $Q$ are matrices satisfying (\ref{e888}), their squares are scalar matrices and $P$ is in the normal Jordan form.

\textit{Case 1}:   One of $P, Q$ is nilpotent, say $P^2=0$. Set $J=\begin{pmatrix}0&1\\0&0\end{pmatrix}$. Then we can take $P$ in the form 
\[
P=\begin{pmatrix}P'&0\\0&0\end{pmatrix}\mbox{ where }P'=\diag\{\underbrace{J,\ld,J}_r\}.
\]
Easy calculation using (\ref{e888}) show that $P'$ must be the whole of $P$.  It will be more convenient to write 
\begin{equation}\label{e777}
P=\begin{pmatrix}0&I_r\\0&0\end{pmatrix}\mbox{ where }k=2r.
\end{equation}

In the above formula zeroes are zero $r\times r$-matrices. We also write $Q=\begin{pmatrix}Q_{11}&Q_{12}\\Q_{21}&Q_{22}\end{pmatrix}$ where each $Q_{ij}$ is an $r\times r$-matrix. If we plug $P$ and $Q$, as just above, to (\ref{e888}), we will obtain 
\[
\begin{pmatrix}Q_{21}&Q_{11}+Q_{22}\\0&Q_{21}\end{pmatrix}=\begin{pmatrix}-I_r&0\\0&-I_r\end{pmatrix}.
\]
Then
\[
Q=\begin{pmatrix}Q_{11}&Q_{12}\\-I_r&-Q_{11}\end{pmatrix}.
\]
To further simplify $u(X)=Q\ot C$, we can conjugate this matrix by any matrix whose conjugation maps $u(g)=I_{2k}\ot A$ and $u(G)=I_{2k}\ot B$  to their scalar multiples (by $\pm 1$) and $u(x)$ to $u(x)+\mu u(g)$ (see Lemma \ref{liso_inn}). If one writes the conjugating matrix as
\[
T=T_I\ot I+T_A\ot A+T_B\ot B+T_C\ot C,
\]
then the condition $T(I_{2k}\ot A)T^{-1}=\pm I_{2k}\ot A$ and $T(I_{2k}\ot B)T^{-1}=\pm I_{2k}\ot B$ implies that $T$ must be one of $T_I\ot I,T_A\ot A,T_B,\ot B, T_C\ot C$. If we now conjugate $u(x)=P\ot C$ by one of these matrices, then, first of all we must have $\mu=0$ and also, depending on what matrix has been chosen, $T_IP=PT_I$, $T_AP=-PT_A$, $T_BP=-PT_B$ and $T_CP=PT_C$. 

In the case where $TP=PT$, using (\ref{e777}), we find that $T=\begin{pmatrix}T_{11}&T_{12}\\0&T_{11}\end{pmatrix}$. In the case $TP=-PT$, we have $T=\begin{pmatrix}T_{11}&T_{12}\\0&-T_{11}\end{pmatrix}$. Here $T_{11}\in\GL_r$ and $T_{12}\in M_r(\FF)$. Thus in both cases we can conjugate first by a matrix where $T_{11}=I_r$ and the conjugate by $T=\begin{pmatrix}T_{11}&0\\0&\pm T_{11}\end{pmatrix}$.

Before we do this, we remember that $Q^2=\alpha I_k$ for some $\alpha\in\FF$. Hence,
\[
\alpha I_k=\begin{pmatrix}Q_{11}^2-Q_{12}&Q_{11}Q_{12}-Q_{12}Q_{11}\\0&Q_{11}^2-Q_{12}\end{pmatrix}.
\]
As a result, $Q_{12}=Q_{11}^2-\alpha I_r$ and 
\[
Q=\begin{pmatrix}Q_{11}&Q_{11}^2-\alpha I_r\\-I_r&-Q_{11}\end{pmatrix}.
\]
If we conjugate this matrix by $T=\begin{pmatrix}I_r&T_{12}\\0&\sigma I_r\end{pmatrix}$, where $\sigma=\pm 1$, we will obtain the following
\[
TQT^{-1}=\begin{pmatrix}Q_{11}-T_{12}&-\sigma Q_{11}T_{12}+\sigma T_{12}^2+\sigma Q_{11}^2-\alpha\sigma I_r-\sigma T_{12}Q_{11}\\-\sigma I_r&T_{12}-Q_{11}\end{pmatrix}.
\]
If we set $T_{12}=Q_{11}$ then
\[
TQT^{-1}=\begin{pmatrix}0&-\alpha\sigma I_r\\-\sigma I_r&0\end{pmatrix}.
\]

\begin{proposition}\label{pMixedDT2} If $D(T_2)$ acts on $M_m(\FF)$, so that $u(g)$ does not commute with $u(G)$ and $u(x)$ is nilpotent then $m=2k$ and the action is isomorphic to precisely one inner action where 
\[
u(g)=I_k\ot A,\: u(G)=I_k\ot B,\: u(x)=\begin{pmatrix}0&I_r\\0&0\end{pmatrix}\ot C,\:u(X)=\begin{pmatrix}0&\tau I_r\\-I_r&0\end{pmatrix}\ot C,
\]
for some $\tau\in\FF$.
\end{proposition}
\begin{proof}
This follows if in the preceding argument we use conjugation  by a matrix $T\ot I_2$ where $T=\begin{pmatrix}I_r&T_{12}\\0& I_r\end{pmatrix}$. Also,  our argument about possible shapes for the conjugating matrix $T$ shows that different $\tau$ produce non-isomorphic actions.
\end{proof}

\textit{Case 2}:   $P^2=\alpha I_k$, $Q^2=\beta I_k$, $\alpha,\beta\ne 0$. Choose $\xi\in\FF$ such that $xi^2=\alpha$. Then after conjugation by a matrix $T\ot I_2$, which does not change the actions by $u(g)$ and $u(G)$, we can reduce $P$ to $P=\begin{pmatrix}\xi I_r&0\\0& -\xi I_s\end{pmatrix}$. Let apply (\ref{e888}) to $P$ as just above and $Q=\begin{pmatrix}X&Y\\Z&U\end{pmatrix}$, where the blocks have the sizes $r\times r$, $r\times s$, $s\times r$, and $s\times s$. Then
\[
\begin{pmatrix}\xi I_r&0\\0& -\xi I_s\end{pmatrix}\begin{pmatrix}X&Y\\Z&U\end{pmatrix}+\begin{pmatrix}X&Y\\Z&U\end{pmatrix}\begin{pmatrix}\xi I_r&0\\0& -\xi I_s\end{pmatrix}=\begin{pmatrix} -I_r&0\\0& -I_s\end{pmatrix}.
\]

In this case, $2\xi X=-I_r$,  $2\xi Y=-I_s$, and we can assume 
\[
Q=-\frac{1}{2\xi}\begin{pmatrix}I_r&Y\\Z&-I_s\end{pmatrix}.
\]
Let us also remember $Q^2=\beta I_k$, were $\beta\ne 0$. Then
\[
\beta I_m=\frac{1}{4\alpha}\begin{pmatrix}I_r+YZ&0\\0&ZY+I_s\end{pmatrix},\mbox{ hence }YZ=(4\alpha\beta -1)I_r, ZY=(4\alpha\beta -1)I_s.
\]
Unless $4\alpha\beta -1=0$, it follows that $r=s$ and $YZ=ZY=\dfrac{1}{4\alpha\beta -1}\:I_r$. In this latter case, any further reductions can only be provided by the matrices of the form $T=\begin{pmatrix}K&0\\0&L\end{pmatrix}$ or $U=T\begin{pmatrix}0&I_r\\I_r&0\end{pmatrix}$. We have
\begin{equation}\label{e666}
\begin{pmatrix}K&0\\0& L\end{pmatrix}\begin{pmatrix}I_r&Y\\Z&-I_r\end{pmatrix}\begin{pmatrix}K^{-1}&0\\0&L^{-1}\end{pmatrix}=\begin{pmatrix}I_r&KYL^{-1}\\LZK^{-1}& - I_r\end{pmatrix}.
\end{equation}
One can choose $K$ and $L$ so that $KYL^{-1}=I_r$. Now, in general,  $(LZK^{-1})(KYL^{-1})=\dfrac{1}{4\alpha\beta -1}\:I_r$. Thus, with the latter choice of $K$ and $L$, we should have and so $(LZK^{-1})=\dfrac{1}{4\alpha\beta -1}\:I_r$ and $Q=\begin{pmatrix}I_r&I_r\\\frac{1}{4\alpha\beta -1}\:I_r& -I_r\end{pmatrix}$. As a result, in the case where $4\alpha\beta\ne 1$, we have $k=2r$ and, for a choice of $xi$ with $xi^2=\alpha$,

\begin{gather*}
u(g)=I_k\ot A,\:u(G)=I_k\ot B,\:u(x)=\xi\begin{pmatrix}I_r&0\\0& -I_r\end{pmatrix}\ot C,\\ u(X)=\begin{pmatrix}I_r&I_r\\\frac{1}{4\alpha\beta -1}\:I_r& -I_r\end{pmatrix}\ot C.
\end{gather*}

 It is easy to see that for different choices of of $\xi$ and $\beta$, the actons parametrized by these constants are not isomorphic (remember $\alpha=\xi^2$).
 
 Now let us consider the remaining case where $YZ=ZY=0$. Using our calculation in (\ref{e666}), which is true also when $r\ne s$, we will reduce $Q$ to the form $Q= \begin{pmatrix}I_r&KYL^{-1}\\LZK^{-1}& - I_s\end{pmatrix}$ where now 
 \begin{equation}\label{e6666}
 (KYL^{-1})(LZK^{-1})=(KYL^{-1})(LZK^{-1})=0.
\end{equation} 
  For appropriate $K$ and $L$, we can assume $KYL^{-1}=\begin{pmatrix}I_t&0_{t,s-t}\\0_{r-t,t}& 0_{r-t,s-t}\end{pmatrix}$, for some $t\le r,s$. From (\ref{e6666}), we then find $LZK^{-1}=\begin{pmatrix}0_{t,t}&0_{t,r-t}\\0_{s-t,t}& \Xi\end{pmatrix}$, where $\Xi$ is a rectangular $(s-t)\times(r-t)$-matrix. In this case, we still have $\xi$ as one of the parameters of the action, but instead of one parameter $\beta$, we have an $(s-t)\times(r-t)$-array of parameters $\Xi$. So if we set $Y(r,s,t)=\begin{pmatrix}I_t&0_{t,s-t}\\0_{r-t,t}& 0_{r-t,s-t}\end{pmatrix}$ and $ Z(r,s,\Xi)=\begin{pmatrix}0_{t,t}&0_{t,r-t}\\0_{s-t,t}& \Xi\end{pmatrix}$, then the last family of the actions is given by
  
\begin{gather*}
u(g)=I_k\ot A,\:u(G)=I_k\ot B,\:u(x)=\xi\begin{pmatrix}I_r&0\\0& -I_s\end{pmatrix}\ot C,\\ u(X)=\begin{pmatrix}I_r&Z(r,s,\Xi)\\Y(r,s,t)& -I_s\end{pmatrix}\ot C.
\end{gather*}

\begin{proposition}\label{pMixedDT2_nonD} If $D(T_2)$ acts on $M_m(\FF)$, so that $u(g)$ does not commute with $u(G)$ and $u(x)$ is not nilpotent then $m=k$ and the action is isomorphic to an inner action where
 
\begin{gather*}
u(g)=I_k\ot A,\:u(G)=I_k\ot B,\:u(x)=\xi\begin{pmatrix}I_r&0\\0& -I_s\end{pmatrix}\ot C,\\ u(X)=\begin{pmatrix}I_r&Z(r,s,\Xi)\\Y(r,s,t)& -I_s\end{pmatrix}\ot C.
\end{gather*}
 
Two such actions, with parameters $r,s,t,\xi, \Xi$ and $r',s',t',\xi',\Xi'$, are isomorphic if and only if $r=r'$, $s=s'$, $t=t'$, $\xi=\xi'$ and there is $U\in \GL_s(\FF)$ such that $\Xi'=U\Xi U^{-1}$.
\end{proposition}
Thus, we have obtained a full classification of non-elementary actions of $D(T_2)$ on $M_n(\FF)$, where $\FF$ is an algebraically closed field of characteristic zero. 

\subsection{General elementary actions of the Drinfeld double}\label{ssGEADD}

We now examine the elementary actions of the Drinfeld doubles $D(T_n(\omega))$ in the case $n\ge 2$ on the algebras of the form $\cA=\End V$, $V$ a finite-dimensional vector space over the base field $\FF$, containing $n$th roots of 1. Let $m=\dim V$. Note that by  (\ref{e_main for_DD_ne}) we always have $u(x)u(X)-\omega u(X)u(x)=I+\lambda u(g)u(G)$. Because the grading by $\fG$ is elementary, we have 
\begin{equation}\label{eVvp}
V=\oplus_{\vp\in\fG}V_\vp
\end{equation}
and
\[
u(h)=\sum_{\vp\in\fG}\vp(h)\:\pi_\vp\mbox{ for any }h\in\cG.
\]
Here we denote by $\pi_\vp$ the projection of $V$ onto $V_\vp$.
 
Now let $\chi\in\fG$ be a uniquely defined character such that $\chi(g)=\chi(G)=\omega^{-1}$ Then, as we know, 
\[
u(h)u(x)u(h)^{-1}=\chi(h)u(x)\mbox{ and }u(h)u(X)u(h)^{-1}=\chi^{-1}(h)u(X),
\]
for all $h\in\cG$. It then follows that $u(x)\in \cA_\chi$ and $u(X)\in \cA_{\chi^{-1}}$. Any graded component $\cA_\delta$ of $\cA$ consists of linear transformations $a$, satisfying $a(V_\psi)\subset V_{\delta\psi}$, for any $\psi\in\fG$. If $\cT$ is the cyclic subgroup of order $n$ generated by $\chi$, the space $V$ is split as the sum of  subspaces $V_{\bar{\vp}}$, $\bar{\vp}\in \fG/\cT$,  which are invariant under the action of $u(g), u(G), u(x), u(X)$. We have 
\[
V_{\bar{\vp}}=\sum_{\psi\in\bar{\vp}}V_\psi
\]
If we fix any $\vp\in\bar{\vp}$ and set $V_\vp^k=V_{\vp\chi^k}$, for $k=0,1,\ld,n-1$, then
\[
u(x):V_\vp^k\to V_\vp^{k+1}\mbox{ and }u(X):V_\vp^k\to V_\vp^{k-1}.
\]
We know that there exist $\alpha,\beta\in\FF$ such that $u(x)^n=\alpha I_m$, $u(X)^n=\beta I_m$. 

In what follows we restrict ourselves to the case when one of $u(x)$, $u(X)$ is nonsingular. Let us assume $\alpha\ne 0$. In this case, as we've seen earlier in Section  \ref{ssMFCGGL}, for any $\vp\in\wG$ fixed, the dimensions of all $V_\vp^k$ are the same. Let us set 
\[
f_k=\pi_{\vp^{k+1}}\circ u(x)\circ\pi_{\vp^k}\mbox{ and }g_k= \pi_{\vp^{k-1}}\circ u(X)\circ\pi_{\vp^k}.\]
Now 
\[
f_{k-1}\circ\cdots\circ f_{k+1}\circ f_k=\alpha\:\pi_{\vp^k}.
\]
 So all $f_k$ are isomorphisms. Let $\dim V_\vp^k=r$. If we choose a basis $\cB^0$ in $V_\vp^0$ and set $\cB^k=f_k(\cB^{k-1})$, for any $k=1,2,\ld,n-1$, then we will obtain a basis $\cB=\cB^0\cup\ldots\cup\cB^{n-1}$ such that the matrix for $u(x)\vert_{V_{\bar{\vp}}}$ will have the form 
\[
\begin{pmatrix} 0&0&0&\cdots&0&\alpha I_r\\ I_r&0&0&\cdots&0&0\\0&I_r&0&\cdots&0&0\\\cdots&\cdots&\cdots&\cdots&\cdots\\0&0&0&\cdots&I_r&0\end{pmatrix}.
\]
To determine $u(X)$, note that the subspaces $V_{\bar{\vp}}$ are invariant under all $u(x)$, $u(X)$, $u(g)$, $u(G)$. This allows us to use the relations for these matrices derived earlier, for the restrictions to each of these subspace. For instance, we will have 
\[
g_{k+1}\circ\cdots\circ g_{k-1}\circ g_k=\beta\:\pi_{\vp^k}.
\]

Let us use $u(x)u(X)-\omega u(X)u(x)=I-\lambda u(g)u(G)$, restricted to $V_{\bar{\vp}}$. We have

\begin{gather*}
\sum_{k=0}^{n-1}(f_{k-1}\circ g_k-\omega g_{k+1}\circ f_k)=\sum_{k=0}^{n-1}(1-\lambda\vp\chi^k(gG))\pi_{\vp^k}\\\mbox{ or  }\\
\sum_{k=0}^{n-1}(f_{k-1}\circ g_k-\omega g_{k+1}\circ f_k)=\sum_{k=0}^{n-1}(1-\lambda\omega^{-2k}\vp(gG))\pi_{\vp^k}
\end{gather*}

Let us set $\rho_k=(1-\lambda\omega^{-2k}\vp(gG))$. Then for each $k=0,1,\ld,n-1$, we will have 
\[
f_{k-1}\circ g_k-\omega g_{k+1}\circ f_k=\rho_k\pi_{\vp^k}.
\]
 This relation allows us to express all $g_k$ in terms of one of them, once $f_0,\ld,f_{k-1}$ is known. For any $k=0,1,\ld,n-1$ we will have 
\[
g_k=\omega f_{k-1}^{-1}\circ g_{k+1}\circ f_k+\rho_k\,f_{k-1}^{-1}.
\]

Using the above basis $\cB$ for $V_\vp$, let us denote by $(v_{ij})$ the block diagonal matrix for $u(X)$ in this basis and $\left(u_{ij}\right)$ the matrix for $u(x)$. 
We have
\[
u(X)=\begin{pmatrix} 0&v_{12}&0&\cdots&0\\ 0&0&v_{23}&\cdots&0\\\cdots&\cdots&\cdots&\cdots&\cdots\\0&0&0&\cdots&v_{n-1,n}\\v_{n1}&0&0&\cdots&0\end{pmatrix}.
\]
Then
\[
v_{k-1,k}=\omega u_{k,k-1}^{-1}v_{k,k+1}u_{k+1,k}+\rho_k\,u_{k,k-1}^{-1}.
\]
We already know that the only nonzero blocks in this matrix are $u_{k+1,k}$, for $k=0,\ld,n-1$, and $u_{k+1,k}=I_r$ if $k=1,\ld,n-1$ while $u_{0,n}=\alpha I_r$. As a result,  we will obtain the recurrent matrix relation
\[
v_{k-1,k}=\omega v_{k,k+1}+\rho_kI_r.
\]
for $k=1,\ld,n-2$
and

\begin{gather*}
v_{n-1,0}=\alpha^{-1}(\omega v_{0,1}+\rho_nI_r)\\
v_{n-2,n-1}=\alpha(\omega v_{n-1,0}+\rho_{n-1}I_r)
\end{gather*}

This recurrent relation enables us to express all blocks of the matrix of the restriction $u(X)$ to  $V_{\bar{\vp}}$ by a single one, say, $v_{0,n-1}$. Moreover, if we conjugate all matrices in question by any $\diag\{T,T,\ld,T\}$, where $T$ is a non-singular square matrix of order $r$, then the matrices for $u(g), u(G), u(x)$ remain the same while we may assume that  $v_{0,n-1}$ is a matrix in a Jordan normal form.

\section{Actions of $H=u_q(\Sl_2)$}\label{suqsl2}

We adopt the following presentation of $H=u_q(\Sl_2)$ from \cite{AS}.  $H$ is generated by the elements  $a,x,y$ such that \begin{eqnarray}\label{ecr} 
 a^n&=&1,\; x^n=y^n=0\nonumber\\
 ax&=&\omega^2 xa,\; ay=\omega^{-2}ya,\; xy-yx=a-a^{-1}.
\end{eqnarray}

Here $n$ is an odd number, $n\ge 3$.
 
 The coproducts are given by 
\[
\Delta x=x\ot a+1\ot x,\;\Delta y=y\ot 1+a^{-1}\ot y, \Delta a=a\ot a.
\]
Here $\omega$ is a primitive $n$th root of 1, $n$ an odd number.

Suppose $H$ acts on a matrix algebra $\cA$. Then the action is inner, via a convolution invertible map $u:H\to \cA$ and given by the formulas below, where $A$ is any element of $\cA$.
\begin{eqnarray}\label{eaction}
a\ast A&=&u(a)Au(a)^{-1}\nonumber\\
x\ast A&=&u(x)Au(a)^{-1}-Au(x)u(a)^{-1}\nonumber\\
y\ast A&=&u(y)A-u(a)^{-1}Au(a)u(y)
\end{eqnarray}
\begin{proposition}\label{pRelUQSL2} The elements $u(a), u(x), u(y)\in \cA$ on the right hand sides of  the equations in \emph{(\ref{eaction})} can be so chosen that
\begin{equation}\label{eu1} 
 u(a)^n=\theta 1,
 \end{equation}
 \begin{equation}\label{eu2} 
 u(a)u(x)=\omega^2 u(x)u(a),
  \end{equation}
 \begin{equation}\label{eu3} 
 u(a)u(y)=\omega^{-2}u(y)u(a),
  \end{equation}
 \begin{equation}\label{eu4} 
 u(x)^n=\mu 1,\;u(y)^n=\nu 1
  \end{equation}
 \begin{equation}\label{eu5}  
 u(x)u(y)-u(y)u(x)=u(a)-\tau u(a)^{-1},
 \end{equation}
where $\theta\ne 0, \tau, \mu,\;\nu\in\FF$.
\end{proposition}

\begin{proof}Note that actually $H$ is an algebra of the type described in the Definition \ref{dPHA}, if we set $\cG=(a)$, $x_1=xa^{-1}$, $a_1=a^{-1}$, $x_2=y$, $a_2=a^{-1}$, $\chi_1(a)=\omega^2$, $\chi_2(a)=\omega^{-2}$. We also have $N_1=N_2=n$, $\mu_1=\mu_2=0$, hence $x_1^n=x_2^n=0$. Also,
 
\begin{gather}
x_2x_1-\chi_1(a_2)x_1x_2=yxa^{-1}-\chi_1(a)^{-1}xa^{-1}y\\
=yxa^{-1}-\omega^{-2}xa^{-1}y=yxa^{-1}-xya^{-1}\\
=(a^{-1}-a)a^{-1}=-(1-a_1a_2)
\end{gather}

It follows that $\lambda_{12}=-1$. Notice also that $\chi_1(a_2)\chi_2(a_1)=1$, as needed in Definition \ref{dPHA}.
 Now one can apply Theorems \ref{pCR}, \ref{lSC}, \ref{pij-ji} to get our needed relations (\ref{eu1}),  (\ref{eu2}),  (\ref{eu3}),  (\ref{eu4}),  (\ref{eu5}). The main tool here is Proposition \ref{p1aa1}.
 
 For example, for any $g\in\cG$, we have
 
 \begin{gather*}
 u(g)u(x_1)u(g)^{-1}=\chi_1(g)u(x_1)\mbox{ so that }u(a)u(x_1)u(a)^{-1}=\omega^2u(x_1)
  \end{gather*}
  
Using Proposition \ref{p1aa1}, we get

 \begin{gather*}
u(a)u(x)u(a)^{-2}-\lambda u(a)^{-1}=\omega^2u(x)u(a)^{-1}-\omega^2\lambda u(a)^{-1}\\
u(a)u(x)u(a)^{-1}=\omega^2u(x)+\lambda(1-\omega^2) I,
  \end{gather*}
  
 as claimed.
 
 As another example,
 
\begin{gather*}
u(x_2)u(x_1)-\chi_1(a_2)u(x_1)u(x_2)=-1+\zeta_{12}u(a_1)u(a_2)\\
u(y)u(x)u(a)^{-1}-\omega^{-2}u(x)u(a)^{-1}u(y)=-1+\zeta_{12}u(a)^{-1}\\
u(y)u(x)-u(x)u(y)=-u(a)+\zeta_{12}u(a)^{-1}\\
u(x)u(y)-u(y)u(x)=u(a)-\zeta_{12}u(a)^{-1},
\end{gather*}
 
as needed.
\end{proof}
\subsection{Book Hopf algebra}\label{ssBHA}

One more Hopf algebra, which is a part of the classification of Hopf algebras of dimension $p^3$, $p$ is an odd prime number, in \cite{AS},  is the so called \textit{Book Hopf algebra} $\textbf{h}(q,m)$, where $q$ is $p$-th root of $1$,  and $m\in\ZZ_p\setminus\{ 0\}$. In terms of generators and defining relations generators and defining relations this algebra is given by
\[
\textbf{h}(q,m)=\langle a,x,y\:|\: axa^{-1}=qx,\:aya^{-1}=q^my,\: a^p=1,\: x^p=0,\:y^p=0,\: xy-yx=0\rangle.
\]
The coproduct is given by the following
\[
\Delta(a)=a\ot a,\: \Delta(x)=x\ot a+1\ot x,\: \Delta(y)=y\ot 1+a^m\ot y.
\]
The following is very similar by the statement and the proof  to Proposition \ref{pRelUQSL2}.

Suppose $H$ acts on a matrix algebra $\cA$. Then the action is inner, via a convolution invertible map $u:H\to \cA$ and given by the formulas below, where $A$ is any element of $\cA$.
\begin{eqnarray}\label{eaction_book}
a\ast A&=&u(a)Au(a)^{-1}\nonumber\\
x\ast A&=&u(x)Au(a)^{-1}-Au(x)u(a)^{-1}\nonumber\\
y\ast A&=&u(y)A-u(a)^mAu(a)^{-m}u(y)
\end{eqnarray}
\begin{proposition}\label{pBook} The elements $u(a), u(x), u(y)\in \cA$ on the right hand sides of  the equations in \emph{(\ref{eaction_book})} can be so chosen that
\begin{equation}\label{eu1_b} 
 u(a)^n=\theta 1,
 \end{equation}
 \begin{equation}\label{eu2_b} 
 u(a)u(x)=q u(x)u(a),
  \end{equation}
 \begin{equation}\label{eu3_b} 
 u(a)u(y)=q^mu(y)u(a),
  \end{equation}
 \begin{equation}\label{eu4_b} 
 u(x)^n=\mu 1,\;u(y)^n=\nu 1
  \end{equation}
 \begin{equation}\label{eu5_b}  
 u(x)u(y)-u(y)u(x)=\tau u(a)^{m},
 \end{equation}
where $\theta\ne 0, \tau, \mu,\;\nu\in\FF$.
\end{proposition}

\begin{proof}
Again, in the same way as in Proposition \ref{pRelUQSL2}, we note that the algebra in question is an algebra from Definition \ref{dPHA}. After this, our relations (\ref{eu1_b}), (\ref{eu2_b}), (\ref{eu3_b}), (\ref{eu4_b}), (\ref{eu5_b}) follow with the use of Theorems \ref{pCR}, \ref{lSC}, \ref{pij-ji} and Proposition \ref{p1aa1}.
\end{proof}

\subsection{Actions on $M_2(\FF)$}\label{ssC23}

As an example of the results in the previous section, we consider the action of $H=u_q(\Sl_2)$, on $\cA=M_2(\FF)$. First of all, there is a basis of the underlying vector space $\FF^2$ where $u(a)=\lambda\left(\begin{array}{cc} 1&0\\0&\omega^k\end{array}\right)$, for $\lambda\in\FF$ satisfying $\lambda^n=1$, $1\le k<n$. If $u(x)=\tau\left(\begin{array}{cc} a&b\\c&d\end{array}\right)$, then
\[
u(a)u(x)u(a)^{-1}=\omega^2 u(x),\mbox{ that is } \left(\begin{array}{cc} a&\omega^{n-k}b\\\omega^k c&d\end{array}\right)=\left(\begin{array}{cc} \omega^2 a&\omega^2 b\\\omega^2 c&\omega^2 d\end{array}\right).
\]
Since we assume $n\ge 3$, we must have $a=d=0$. We also must have $\omega^{n-k}b=\omega^2b$ and $\omega^kc=\omega^2c$. Thus if none of $k$, $n-k$ equals 2, the action of $x$ is trivial (this also follows by Proposition \ref{p99}). Since the computation for $u(y)$ is similar, in the case none of $k, n-k$ equals 2, the action of $H$ is purely a group action, so is  just a grading by $\fG\cong\ZZ_n$. Since also $n\ne 4$, only one of $k$, $n-k$ equals 2. Let us first look at the case $k=2$. In this case
\[
u(x)=\left(\begin{array}{cc} 0&0\\p&0\end{array}\right), \; u(y)=\left(\begin{array}{cc} 0&q\\0&0\end{array}\right),
\]
for some $p,q\in\FF$. Clearly they satisfy all (\ref{eu1}) to (\ref{eu4}), so we need only to see what restrictions are imposed by (\ref{eu5}). We have 
\[
\left(\begin{array}{cc} -qp&0\\0&pq\end{array}\right)=\left(\begin{array}{cc} \lambda-\lambda^{-1}\tau&0\\0&\lambda\omega^2-\lambda^{-1}\tau\omega^{n-2}\end{array}\right)
\]
Thus 
\[
\tau=\lambda^2\frac{1+\omega^2}{1+\omega^{n-2}}\mbox{ and } pq=\lambda\frac{\omega^2-\omega^{n-2}}{1+\omega^{n-2}}.
\]
 Similar is the case $n-k=2$.

If $n=4$ then $k=2$ and we have 
\[
u(x)=\left(\begin{array}{cc} 0&p\\q&0\end{array}\right), \; u(y)=\left(\begin{array}{cc} 0&r\\s&0\end{array}\right),
\]
for some $p,q,r,s\in\FF$. Again, we need to check (\ref{eu5}). It follows from (\ref{eu5}) that 
\[
\left(\begin{array}{cc} ps-qr&0\\0&qr-ps\end{array}\right)=(\lambda-\lambda^{-1}\tau)\left(\begin{array}{cc} 1&0\\0&\omega^2\end{array}\right),
\]
Since $\omega^2=-1$, it follows that any $u(x), u(y)$ can be chosen, and  (\ref{eu5}) will be satisfied with $\tau=\lambda-ps+qr$. This  enough freedom to produce different actions of $H$ on $\cA$. It would be good to determine the equivalence classes of actions.

\subsection{Actions of Drinfeld doubles of the Taft algebra via the actions of  $u_q(sl_2)$}\label{ssDDUQ}

 An alternate approach to the results in Section \ref{sDD} can be given as follows (see \cite[Corollary 4.8]{SM}, we use the notation from \cite[IX.6]{K}).
First, it is well-known that $u_{q}(sl_{2})$ is a Hopf homomorphic image of $D(H)$ ,
for $H=T_{n^{2}}(\omega)$ , by setting $\omega=q^{2}$ and defining $\Phi$ : $D(H)\rightarrow u_{q}$ as follows:
$G\mapsto K$ , $g\mapsto K^{-1}$ , $x\mapsto F$ , and $X\mapsto\overline{E}K^{-1}$ $:=-(q-q^{-1})E$.

Moreover $\Ker(\Phi)=D(H)k\mathcal{G}^{+}$, where $\mathcal{G}$ is the cyclic subgroup generated by $gG$. 

\begin{corollary}\label{c4.8}
 Each action of $u_q(sl_{2})$ on an algebra $\cA$ lifts to an action of the Drinfeld double $D(T_q)$ on $\cA$ via the homomorphism shown above. Thus any action of  $u_q(sl_{2})$ on $M_n(\FF)$can be obtained from an \emph{elementary} action of $D(T_q)$ described in Section \ref{ssGEADD}, with an additional provision that in \emph{(\ref{eVvp})} the summation is taken over the characters $\vp\in\fG$ such that $\vp(gG)=1$.
\end{corollary}

Note that the converse is not true: no division action of the Drinfeld double $D(T_q)$ comes as a lifting of an action of the $u_q(sl_{2})$, because $G$ and $g$ do not act as inverses of each other. For more specific examples see the formulas at the end of Section \ref{ssADTN}).

\end{document}